\documentclass[a4paper]{amsart}

\usepackage{hyperref}
\usepackage[abbrev]{amsrefs}
\usepackage{amsfonts}
\usepackage{amssymb}
\usepackage{array}
\usepackage{graphicx}
\usepackage{overpic}
\usepackage[all]{xypic}

\BibSpec{article}{%
+{}{\sc \PrintAuthors} {author}
+{,}{ \textrm} {title}
+{,}{ \textit} {journal}
+{,}{ } {volume}
+{}{ \IfEmptyBibField{journal}{}{\PrintDate{date}}} {transition}
+{,}{ } {pages}
+{,}{ } {status}
+{.}{} {transition}
+{}{ Preprint\IfEmptyBibField{date}{}{ \PrintDate{date}}, available at \tt} {eprint}
+{.}{} {transition}
}

\BibSpec{book}{%
+{}{\sc \PrintAuthors} {author}
+{,}{ \textit} {title}
+{.}{ } {series}
+{,}{ } {volume}
+{.}{ } {publisher}
+{,}{ } {place}
+{,}{ } {date}
+{.}{} {transition}
}

\BibSpec{collection.article}{
+{} {\sc \PrintAuthors} {author}
+{,} { \textrm } {title}
+{,} { in \it \PrintConference} {conference}
+{,} { pp.~} {pages}
+{.} {\PrintBook} {book}
+{,} { \PrintDateB} {date}
+{.} {} {transition}
}

\BibSpec{innerbook}{
+{.} { \emph} {title}
+{.} { } {part}
+{:} { \emph} {subtitle}
+{.} { } {series}
+{,} { } {volume}
+{.} { Edited by \PrintNameList} {editor}
+{.} { Translated by \PrintNameList}{translator}
+{.} { \PrintContributions} {contribution}
+{.} { } {publisher}
+{.} { } {organization}
+{,} { } {address}
+{,} { \PrintEdition} {edition}
+{,} { \PrintDateB} {date}
+{.} { } {note}
}

\allowdisplaybreaks[3]

\newcommand{\PP}{\mathbb{P}}
\newcommand{\FF}{\mathbb{F}}
\newcommand{\CC}{\mathbb{C}}
\newcommand{\ZZ}{\mathbb{Z}}
\newcommand{\RR}{\mathbb{R}}
\newcommand{\QQ}{\mathbb{Q}}
\newcommand{\NN}{\mathbb{N}}

\renewcommand{\cD}{\mathcal{D}}
\newcommand{\cF}{\mathcal{F}}

\newcommand{\cHX}{\mathcal{H}_\cX}
\newcommand{\cHY}{\mathcal{H}_Y}
\newcommand{\cHZ}{\mathcal{H}_\cZ}
\newcommand{\cHopp}{\mathcal{H}^{\text{\rm opp}}}

\newcommand{\ev}{\mathrm{ev}}
\newcommand{\Nov}{\Lambda}
\newcommand{\NovX}{\Lambda_\cX}
\newcommand{\NovY}{\Lambda_Y}
\newcommand{\NovZ}{\Lambda}
\newcommand{\bt}{\mathbf{t}}
\newcommand{\btau}{\boldsymbol{\tau}}
\newcommand{\bfq}{\mathbf{q}}

\renewcommand{\(}{\left(}
\renewcommand{\)}{\right)}

\DeclareMathOperator{\Res}{Res}
\newcommand{\fun}{\mathbf{1}}

\newcommand{\cC}{\mathcal{C}}
\newcommand{\cX}{\mathcal{X}}
\newcommand{\cY}{\mathcal{Y}}
\newcommand{\cZ}{\mathcal{Z}}

\newcommand{\cLX}{\mathcal{L}_{\mathcal{X}}}
\newcommand{\cLY}{\mathcal{L}_Y}
\newcommand{\cLZ}{\mathcal{L}_\cZ}
\newcommand{\cDX}{\mathcal{D}_{\mathcal{X}}}
\newcommand{\cDY}{\mathcal{D}_Y}
\newcommand{\HorbX}{H^\bullet_{\scriptscriptstyle \text{\rm
      CR}}(\cX;\CC)}
\newcommand{\HorbNovX}{H^\bullet_{\scriptscriptstyle \text{\rm
      CR}}(\cX;\NovX)}
\newcommand{\HorbXNovZ}{H^\bullet_{\scriptscriptstyle \text{\rm
      CR}}(\cX;\NovZ)}
\newcommand{\HorbZ}{H^\bullet_{\scriptscriptstyle \text{\rm
      CR}}(\cZ;\CC)}
\newcommand{\HorbiZ}[1]{H^{#1}_{\scriptscriptstyle \text{\rm
      CR}}(\cZ;\CC)}
\newcommand{\HorbNovZ}{H^\bullet_{\scriptscriptstyle \text{\rm CR}}(\cZ;\NovZ)}
\newcommand{\HY}{H^\bullet(Y;\CC)}
\newcommand{\HNovY}{H^\bullet(Y;\NovY)}
\newcommand{\HYNovX}{H^\bullet(Y;\NovX)}
\newcommand{\HYNovZ}{H^\bullet(Y;\NovZ)}
\newcommand{\HZ}{H^\bullet(\cZ;\CC)}
\newcommand{\correlator}[1]{\left \langle #1 \right \rangle}
\newcommand{\ccorrelator}[1]{\big \langle \!\! \big \langle #1 \big
  \rangle \!\! \big \rangle}
\newcommand{\CCorrelator}[1]{\bigg \langle \!\!\! \bigg \langle #1 \bigg
  \rangle \!\!\! \bigg \rangle}
\newlength{\mybracketspacing}

\DeclareMathOperator{\id}{id}

\newcommand{\T}{\mathbb{T}}
\newcommand{\U}{\mathbb{U}}
\newcommand{\QC}[1]{\underset{#1}{\star}}
\newcommand{\newQC}[1]{\underset{#1}{\circledast}}
\newcommand{\smallQC}{\bullet}
\newcommand{\CR}{\underset{\scriptscriptstyle \text{CR}}{\cup}}
\newcommand{\opp}{{\text{opp}}}

\newcommand{\re}{\mathrm{e}}
\DeclareMathOperator{\Real}{Re}

\theoremstyle{plain}

\newtheorem{theorem}{Theorem}[section]
\newtheorem{lemma}[theorem]{Lemma}
\newtheorem{proposition}[theorem]{Proposition}
\newtheorem{conj}[theorem]{Conjecture}
\newtheorem*{conj*}{Conjecture}
\newtheorem{cor}[theorem]{Corollary}
\newtheorem*{cor*}{Corollary}
\newtheorem*{applemma1}{Lemma A.1}
\newtheorem*{applemma2}{Lemma A.2}

\theoremstyle{definition}

\newtheorem{convassum}[theorem]{Convergence Assumption}
\newtheorem{rem}[theorem]{Remark}
\newtheorem*{definition}{Definition}

\DeclareMathOperator{\NE}{Eff}
\DeclareMathOperator{\Hom}{Hom}

\renewcommand{\cL}{\mathcal{L}}

\newcommand{\Lie}{\operatorname{Lie}}
\def\parfrac#1#2{\frac{\partial #1}{\partial #2}}

\begin{document}

\title[Quantum Cohomology and Crepant Resolutions]{Quantum Cohomology and Crepant Resolutions: A Conjecture}

\author{Tom Coates}
\address{Department of Mathematics \\ Imperial College London \\
  London SW7 2AZ \\ United Kingdom}
\email{tomc@imperial.ac.uk}

\author{Yongbin Ruan}
\address{ Department of Mathematics\\ University of Michigan \\ Ann Arbor MI
  48105 \\ USA }
\email{ruan@umich.edu}


\begin{abstract}
  We give an expository account of a conjecture, developed by
  Coates--Corti--Iritani--Tseng and Ruan, which relates the quantum
  cohomology of a Gorenstein orbifold $\cX$ to the quantum cohomology
  of a crepant resolution $Y$ of $\cX$.  We explore some consequences
  of this conjecture, showing that it implies versions of both the
  Cohomological Crepant Resolution Conjecture and of the Crepant
  Resolution Conjectures of Ruan and Bryan--Graber.  We also give a
  `quantized' version of the conjecture, which determines higher-genus
  Gromov--Witten invariants of $\cX$ from those of $Y$.
\end{abstract}

\maketitle

\section{Introduction}

An orbifold is a space which is locally modelled on quotients of
$\RR^n$ by finite groups.  Orbifolds are a natural class of spaces to
study.  Manifolds and smooth algebraic varieties are orbifolds but
spaces of geometric interest, and particularly those obtained by
quotient constructions, are often orbifolds rather than varieties or
manifolds.  Furthermore many geometric operations, including those
transformations involved in spacetime topology change
\cite{Aspinwall--Greene--Morrison}, treat orbifolds and smooth
varieties on an equal footing.  In this paper we study the quantum
cohomology of orbifolds.

The quantum cohomology of a K\"ahler orbifold $\cX$ is a family of
algebras whose structure constants encode certain \emph{Gromov--Witten
  invariants} of $\cX$.  These Gromov--Witten invariants are
interesting from at least three points of view: \emph{symplectic
  topology}, as they give invariants of $\cX$ as a symplectic
orbifold; \emph{algebraic geometry}, as they give a `virtual count' of
the number of curves in $\cX$ which are constrained to pass through
various cycles; and \emph{physics}, as they give rigorous meaning to
instanton counting in a model of string theory with spacetime $\cX
\times \RR^4$.  In what follows we outline a conjecture which
describes how the quantum cohomology of a Gorenstein orbifold $\cX$ is
related to that of a crepant resolution $Y$ of $\cX$, and explore some
of its consequences.  The conjecture is of interest also from at least
three points of view: Gromov--Witten invariants of orbifolds are
\emph{difficult to compute}, and the conjecture provides tools for
doing this; crepant resolutions are simple examples of
\emph{birational transformations}, and an understanding of how quantum
cohomology changes under birational transformations would be both
interesting and useful; and the conjecture provides a version of the
\emph{McKay Correspondence} which reflects a well-known physical
principle --- that string theory on an orbifold and on a crepant
resolution of that orbifold should be equivalent.

The conjecture, which is described in more detail in \S\ref{sec:conj}
below, was developed by Coates--Corti--Iritani--Tseng
\cite{CCIT:crepant1} and Ruan \cite{Ruan:conjecture}.  Following
Givental, we encode all genus-zero Gromov--Witten invariants of $\cX$
in the germ $\cLX$ of a Lagrangian submanifold in a symplectic vector
space $\cHX$.  This submanifold-germ $\cLX$ has very special geometric
properties (theorem~\ref{thm:cone} below) which make it easy to
determine the quantum cohomology of $\cX$ from $\cLX$
(\S\ref{sec:conetoQC} below).  A similar submanifold-germ $\cLY
\subset \cHY$ encodes all genus-zero Gromov--Witten invariants of the
crepant resolution $Y$.  As $\cLX$ and $\cLY$ are germs of
submanifolds, it makes sense to analytically continue them.  We
conjecture that there is a linear symplectic isomorphism $\U:\cHX \to
\cHY$ such that after analytic continuation of $\cLX$ and $\cLY$ we
have $\U(\cLX) = \cLY$.  This gives, in particular, a conjectural
relationship between the quantum cohomology of $\cX$ and the quantum
cohomology of $Y$.

The idea that the quantum cohomology of $\cX$ should be in some sense
equivalent to the quantum cohomology of $Y$ has been around for a
while now, and is due to Ruan.  He originally conjectured that the
small quantum cohomology of $\cX$ and the small quantum cohomology of
$Y$ --- two families of algebras which depend on so-called quantum
parameters --- become isomorphic after specializing some of the
quantum parameters to particular values.  This specialization may
first require analytic continuation in the quantum parameters.  Ruan's
conjecture is discussed further and revised in \S\ref{sec:Ruan} and
\S\ref{sec:gerbes} below.
Bryan and Graber \cite{Bryan--Graber} recently proposed a refinement
of Ruan's conjecture, applicable whenever $\cX$ satisfies a Hard
Lefschetz condition on orbifold cohomology \cite{CCIT:crepant1}.  They suggest that in this
case the big quantum cohomology algebras of $\cX$ and $Y$ coincide
after analytic continuation and specialization of quantum parameters,
via a linear isomorphism that also matches certain pairings on the
algebras.

As we explain in \S\S\ref{sec:Ruan}--\ref{sec:BG} below, under
appropriate conditions on $\cX$ our conjecture implies something very
like the earlier conjectures of Ruan and Bryan--Graber.  Our
conjecture applies, however, in much greater generality.  This fits
with a general picture developed by Givental: that the
submanifold-germ $\cLX$ often transforms in a simple way under
geometric operations on $\cX$, even when those operations have a
complicated effect on quantum cohomology.  Our conjecture also fits
well with Givental's approach to mirror symmetry.  This was the
essential point in the proof \cite{CCIT:crepant1} of the conjecture
for $\cX = \PP(1,1,2)$ and $\cX = \PP(1,1,1,3)$.  Forthcoming work by
Coates, Corti, Iritani, and Tseng will extend this line of argument,
using mirror symmetry to prove our conjecture for crepant resolutions
of toric orbifolds $\cX$ such that $c_1(\cX) \geq 0$.

An outline of the paper is as follows.  We give introductions to the
cohomology and quantum cohomology of orbifolds in \S\ref{sec:QC}, and
to Givental's framework in \S\ref{sec:cone}.  We state the conjecture
in \S\ref{sec:conj}.  After giving some preparatory lemmas
(\S\ref{sec:properties}), we explain in \S\ref{sec:conetoQC} how to
extract quantum cohomology from the submanifold $\cLX$.  This allows
us to draw conclusions about quantum cohomology from our conjecture.
We do this in the next three sections, proving something like the
Cohomological Crepant Resolution Conjecture in \S\ref{sec:CCRC},
something like Ruan's conjecture in \S\ref{sec:Ruan}, and something
like the Bryan--Graber conjecture in \S\ref{sec:BG}.  We close by
discussing a higher-genus version of the conjecture
(\S\ref{sec:highergenus}) and the role of flat gerbes
(\S\ref{sec:gerbes}).

We should emphasize that most of what follows is a new presentation of
ideas and methods which are already in the literature; in particular
we draw the reader's attention to \citelist{\cite{Givental:symplectic}
  \cite{Barannikov} \cite{CCIT:crepant1} \cite{Ruan:firstconjecture}}.
But we feel that these ideas are important enough to deserve a clear
and accessible expository account.  The main purpose of this article
is to give such an account: we are, of course, entirely responsible
for any mistakes or obscurities that it contains.

\subsection*{Acknowledgements.}  

Both authors are very grateful to Hiroshi Iritani: most of the results
in this paper we either learned from him or developed in conversations
with him.  We would have preferred that he join us as author of this
note, but must respect his wishes in this regard.  T.C. thanks Jim
Bryan, Alessio Corti, Alexander Givental, and Hsian-Hua Tseng for
useful discussions; and the Royal Society and the Clay Mathematics
Institute for financial support.  Y.R. thanks Paul Aspinwall for
useful discussions.  This work was partially supported by the National
Science Foundation under grants DMS-0401275 and DMS-0072282.

\section{Orbifold Cohomology and Quantum Cohomology}
\label{sec:QC}

In this section we describe and fix notation for orbifold cohomology,
Gromov--Witten invariants, and quantum cohomology.  The non-expert
reader should be able to follow the rest of the paper after reading
the summary of these topics below; detailed accounts of the theory can
be found in the work of Chen--Ruan \cite{Chen--Ruan:orbifold,
  Chen--Ruan:GW} and Abramovich--Graber--Vistoli \cite{AGV:1,AGV:2}.
We work in the algebraic category, so from now on `orbifold' means
`smooth Deligne--Mumford stack over $\CC$' and `manifold' means
`smooth variety'.

Let $\cZ$ be an orbifold.  The \emph{Chen--Ruan orbifold cohomology}
$\HorbZ$ is the cohomology of the so-called inertia stack of $\cZ$.
If $\cZ$ is a manifold then $\HorbZ$ is canonically isomorphic to the
ordinary cohomology $\HZ$ and so a Chen--Ruan cohomology class can be
represented, via Poincar\'e duality, as a cycle in $\cZ$.  In general
a Chen--Ruan class can be represented as a linear combination of pairs
$(A,[g_A])$ where $A \subset \cZ$ is a connected cycle and $[g_A]$ is
a conjugacy class in the isotropy group of the generic point of $A$.
Chen--Ruan cohomology contains ordinary cohomology as a subspace,
represented by those decorated cycles $(A,[g_A])$ where $g_A$ is the
identity element; if $\cZ$ is a manifold then this subspace is the
whole of $\HorbZ$.  The complementary subspace in $\HorbZ$ spanned by
those decorated cycles $(A,[g_A])$ such that $g_A$ is not the identity
is called the \emph{twisted sector}.  Chen--Ruan cohomology carries a
non-degenerate pairing, the \emph{orbifold Poincar\'e pairing}, which
functions exactly as the usual Poincar\'e pairing except that classes
represented by $(A,[g_A])$ and $(B,[g_B])$ pair to zero unless $[g_A]
= [g_B^{-1}]$.

In what follows we will consider maps $f:\cC \to \cZ$ from orbifold
curves to $\cZ$.  The source curve $\cC$ here may be nodal, and
carries a number of marked points.  We allow $\cC$ to have isotropy at
the marked points and nodes, but nowhere else, and insist that the map
$f$ is \emph{representable}: that it induces injections on all
isotropy groups.  (In particular, therefore, if $\cZ$ is a manifold
then we consider only maps $f:\cC \to \cZ$ from curves with trivial
orbifold structure.)\phantom{.} We take the \emph{degree} of the map
$f:\cC \to \cZ$ to be the degree of the corresponding map between
coarse moduli spaces \cite{Keel--Mori}.  This means the following.  Let
$C$ and $Z$ be the coarse moduli spaces of $\cC$ and $\cZ$
respectively, and let $\bar{f}:C \to Z$ be the map induced by $f$.
Consider the free part 
\[
H_2(Z;\ZZ)_{\rm free} = H_2(Z;\ZZ) / H_2(Z;\ZZ)_{\rm tors} 
\]
of $H_2(Z;\ZZ)$; here $H_2(Z;\ZZ)_{\rm tors}$ is the torsion subgroup
of $H_2(Z;\ZZ)$. The degree $d$ of $f:\cC \to \cZ$, $d \in
H_2(Z;\ZZ)_{\rm free}$, is defined to be the equivalence class of
$\bar{f}_\star [C]$ where $[C]$ is the fundamental class of $C$.

We use correlator notation for the \emph{Gromov--Witten invariants} of
the orbifold $\cZ$, writing
\begin{equation}
  \label{eq:GW}
  \correlator{\delta_1 \psi^{a_1},\ldots,\delta_n
    \psi^{a_n}}^\cZ_{g,n,d} = \correlator{\tau_{a_1}(\delta_1),\ldots,\tau_{a_n}(\delta_n)}_{g,d}
\end{equation}
where $\delta_1, \ldots,\delta_n$ are Chen--Ruan cohomology classes on
$\cZ$; $a_1,\ldots,a_n$ are non-negative integers; and the right-hand
side is defined as on page 41 of \cite{AGV:2}.  If $\cZ$ is a
manifold; $a_1 = \cdots = a_n = 0$; and a very restrictive set of
transversality assumptions hold then \eqref{eq:GW} gives the number of
smooth $n$-pointed curves in $\cZ$ of degree $d$ and genus $g$ which
are incident at the $i$th marked point, $1 \leq i \leq n$, to a chosen
generic cycle Poincar\'e-dual to $\delta_i$ (see
\cite{Fulton--Pandharipande}).  In general, one should interpret
\eqref{eq:GW} as the `virtual number' of possibly-nodal $n$-pointed
orbifold curves in $\cZ$ of genus $g$ and degree $d$ which are
incident to chosen cycles as above.  If any of the $a_i$ are non-zero
then we count only curves which in addition satisfy certain
constraints on their complex structure.  If $\cZ$ is an orbifold but
not a manifold then, as discussed above, the curves we count are
themselves allowed to be orbifolds; the orbifold structure at the
$i$th marked point of the curve is determined by the conjugacy class
$[g_i]$ in a representative $(A_i,[g_i])$ of $\delta_i$.  We write
$\NE(\cZ)$ for the set of possible degrees $d$ in \eqref{eq:GW}, or in
other words for the set of degrees of effective orbifold curves in
$\cZ$.

Henceforth let $\cX$ be a Gorenstein orbifold with projective coarse
moduli space $X$, and let $\pi:Y \to X$ be a crepant resolution.
Assume that the isotropy group of the generic point of $\cX$ is
trivial.  The cohomology and homology groups $H^\bullet(\cX;\QQ)$,
$H_\bullet(\cX;\QQ)$ are canonically isomorphic to $H^\bullet(X;\QQ)$
and $H_\bullet(X;\QQ)$ respectively.  The maps
\begin{align*}
  \pi^\star: H^\bullet(\cX;\QQ) \to H^\bullet(Y;\QQ), &&
  \pi_\star: H_\bullet(Y;\QQ) \to H_\bullet(\cX;\QQ),
\end{align*}
are respectively injective \cite{Beilinson--Bernstein--Deligne} and
surjective, and there is a `wrong-way' map
\[
\pi_!: H^\bullet(Y;\QQ) \to H^\bullet(\cX;\QQ)
\]
defined using Poincar\'e duality.  We refer to elements of $\ker
\pi_!$ as \emph{exceptional classes}.  For an orbifold $\cZ$, we say
that a basis for $H_2(Z;\ZZ)_{\rm free}$ is \emph{positive} if the
degree of any map $f:\cC \to \cZ$ from an orbifold curve is a
non-negative linear combination of basis elements.  Let us fix bases
for homology, cohomology, and orbifold cohomology as follows.  Let
$\beta_1,\ldots,\beta_r$ be a positive basis for $H_2(Y;\ZZ)_{\rm
  free}$ such that
\begin{align*}
  & \text{$\pi_\star \beta_1,\ldots,\pi_\star \beta_s$ is a positive
    basis for $H_2(X;\ZZ)_{\rm free}$,} \\
  & \text{$\beta_{s+1},\ldots,\beta_r$ is a basis for $\ker \pi_\star
    \subset H_2(Y;\ZZ)_{\rm free}$.}
\end{align*}
Choose homogeneous bases $\varphi_0,\ldots,\varphi_N$ for
$H^\bullet(Y;\QQ)$ and $\phi_0,\ldots,\phi_N$ for
$H^\bullet_{\scriptscriptstyle \text{\rm CR}}(\cX;\QQ)$ such that
\begin{align*}
  & \text{$\varphi_0 = \fun_Y$, the identity element in
    $H^\bullet(Y;\QQ)$;} \\
  & \text{$\varphi_1,\ldots,\varphi_r$ is the basis for $H^2(Y;\QQ)$
    dual to   $\beta_1,\ldots,\beta_r$;} \\
  & \text{$\phi_0 = \fun_\cX$, the identity element in
    $H^0(\cX;\QQ)$;} \\
  & \text{$\phi_1,\ldots,\phi_s$ is the basis for $H^2(\cX;\QQ)$ dual
    to $\pi_\star \beta_1,\ldots,\pi_\star \beta_s$;} \\
  & \text{$\phi_1,\ldots,\phi_r$ is a basis for $H^2_{\scriptscriptstyle \text{\rm CR}}(\cX;\QQ)$.}
\end{align*}
Note that $\varphi_i = \pi^\star(\phi_i)$, $1 \leq i \leq s$.  Let
$\varphi^0,\ldots,\varphi^N$ be the basis for $H^\bullet(Y;\CC)$ which
is dual to $\varphi_0,\ldots,\varphi_N$ under the Poincar\'e pairing
$\( \cdot, \cdot \)_Y$, and let $\phi^0,\ldots,\phi^N$ be the basis
for $\HorbX$ which is dual to $\phi_0,\ldots,\phi_N$ under the
orbifold Poincar\'e pairing $\( \cdot, \cdot \)_\cX$.  We will use
Einstein's summation convention for Greek indices, summing repeated
Greek (but not Roman) indices over the range $0,1,\ldots,N$.  For $d
\in \NE(Y)$, let
\begin{align*}
Q^d & = Q_1^{d_1} Q_2^{d_2} \cdots Q_r^{d_r}
&& \text{where} & d & = d_1 \beta_1 + \cdots + d_r \beta_r, \\
\intertext{and for $d \in \NE(\cX)$, let}
  U^{d} & = U_1^{d_1} U_2^{d_2} \cdots U_s^{d_s}
  && \text{where} & d &= d_1 \pi_\star \beta_1 + \cdots +
d_s \pi_\star \beta_s.
\end{align*}
The monomial $Q^d$ is an element of the \emph{Novikov ring for $Y$},
$\NovY = \CC[\![Q_1,\ldots,Q_r]\!]$; the monomial
$U^{d}$ is an element of the \emph{Novikov ring for $\cX$}, $\NovX =
\CC[\![U_1,\ldots,U_s]\!]$.

The \emph{big quantum product} for $\cX$ is a family $\QC{\tau}$ of
algebra structures on $\HorbNovX$, parameterized by $\tau \in
\HorbNovX$, which is defined in terms of Gromov--Witten invariants of
$\cX$.  Let $\tau = \tau_\alpha \phi_\alpha$, and consider the
\emph{genus-zero Gromov--Witten potential} for $\cX$,
\begin{align}
  F_\cX &= \sum_{d \in \NE(\cX)}
  \sum_{n \geq 0}
  \correlator{\tau,\tau,\ldots,\tau}^\cX_{0,n,d}
  { U^{d} \over n!} \notag \\
  \label{eq:FX}
  &= \sum_{\substack{d \in \NE(\cX): \\
      d = d_1 \pi_\star \beta_1 + \cdots + d_s \pi_\star \beta_s}}
  \sum_{ n \geq 0 }
  \correlator{\phi_{\epsilon_1},\ldots,\phi_{\epsilon_n}}^\cX_{0,n,d}
  {U_1^{d_1} \cdots U_s^{d_s} \tau_{\epsilon_1} \cdots
    \tau_{\epsilon_n} \over n!}.
\end{align}
(Recall that we always sum over repeated Greek indices, such as the
$\epsilon_i$ here.)  The Gromov--Witten potential $F_\cX$ is a formal
power series in the variables $\tau_0,\ldots,\tau_N$ and
$U_1,\ldots,U_s$; it is a generating function for
genus-zero Gromov--Witten invariants of $\cX$.  The potential $F_\cX$
determines the big quantum product $\QC{\tau}$ on $\HorbNovX$ via
\begin{equation}
  \label{eq:bigQCX}
  \phi_\alpha \QC{\tau} \phi_\beta =
  {\partial^3 F_\cX \over \partial \tau_\alpha \partial \tau_\beta
    \partial \tau_\gamma} \phi^\gamma.
\end{equation}
We can regard the RHS of \eqref{eq:bigQCX} as a formal power series in
$\tau_0,\ldots,\tau_N$ with coefficients in $\HorbNovX$, and thus
$\QC{\tau}$ gives a family, depending formally on $\tau$, of algebra
structures on $\HorbNovX$.  Similarly, setting $t = t_\alpha
\varphi_\alpha$, the \emph{genus-zero Gromov--Witten potential} for
$Y$,
\begin{align}
  F_Y &= \sum_{d \in \NE(Y)}
  \sum_{n \geq 0}
  \correlator{t,t,\ldots,t}^Y_{0,n,d}
  { Q^d \over n!} \notag \\
  \label{eq:FY}
  &= \sum_{\substack{d \in \NE(Y): \\
      d = d_1 \beta_1 + \cdots + d_r \beta_r}}
  \sum_{ n \geq 0 }
  \correlator{\varphi_{\epsilon_1},\ldots,\varphi_{\epsilon_n}}^Y_{0,n,d}
  {Q_1^{d_1} \cdots Q_r^{d_r} t_{\epsilon_1} \cdots
  t_{\epsilon_n} \over n!}
\end{align}
is a formal power series in the variables $t_0,\ldots,t_N$ and
$Q_1,\ldots,Q_r$.  It determines the \emph{big quantum
  product for $Y$}, which is a family $\QC{t}$ of algebra structures
on $\HNovY$ depending formally on $t \in \HNovY$, via
\begin{equation}
  \label{eq:bigQCY}
  \varphi_\alpha \QC{t} \varphi_\beta =
  {\partial^3 F_Y \over \partial t_\alpha \partial t_\beta
    \partial t_\gamma} \varphi^\gamma.
\end{equation}

The \emph{small quantum products} are algebra structures on
$\HorbNovX$ and $\HNovY$ obtained from the big quantum products
\eqref{eq:bigQCX} and \eqref{eq:bigQCY} by setting $\tau=0$, $t=0$:
\begin{equation}
  \label{eq:smallQC}
  \begin{aligned}
    \phi_\alpha \smallQC \phi_\beta
    &= \sum_{d \in \NE(\cX)}
    \correlator{\phi_\alpha, \phi_\beta, \phi^\gamma}^\cX_{0,3,d}
    U^{d} \phi_\gamma
    & \text{for $\cX$} \\
    \varphi_\alpha \smallQC \varphi_\beta
    &= \sum_{d \in \NE(Y)}
    \correlator{\varphi_\alpha, \varphi_\beta,
      \varphi^\gamma}^Y_{0,3,d}
    Q^d \, \varphi_\gamma
    & \text{for $Y$.}
  \end{aligned}
\end{equation}
The variables $U_1,\ldots,U_s$ and $Q_1,\ldots,Q_r$ hidden here are
the `quantum parameters' described in the introduction.  Setting $Q_1
= \cdots = Q_r = 0$ in \eqref{eq:smallQC} recovers the usual cup
product on $\HY$; setting $U_1 = \cdots = U_s = 0$ gives the
\emph{Chen--Ruan product} on $\HorbX$, which we denote by $\CR$.
Unless otherwise indicated, all products of Chen--Ruan cohomology
classes are taken using $\CR$.

It follows from the Divisor Equation (see \emph{e.g.}
\cite{Bryan--Graber}) that $\phi_\alpha \QC{\tau} \phi_\beta$ depends
on the variables $\tau_1,\ldots,\tau_s,U_1,\ldots,U_s$ only through
the combinations $U_i \re^{t_i}$, $1 \leq i \leq s$, and that
$\varphi_\alpha \QC{t} \varphi_\beta$ depends on the variables
$t_1,\ldots,t_r,Q_1,\ldots,Q_r$ only through the combinations $Q_i
\re^{t_i}$, $1 \leq i \leq r$.  Set
\begin{equation} \label{eq:tworest}
  \begin{aligned}
    \tau_{\rm two} = \tau_1 \phi_1 + \cdots + \tau_s \phi_s, &&
    \tau_{\rm rest} = \tau_0 \phi_0 + \tau_{s+1} \phi_{s+1} + \cdots +
    \tau_N \phi_N, \\
    t_{\rm two} = t_1 \varphi_1 + \cdots + t_r \varphi_r, &&
    t_{\rm rest} = t_0 \varphi_0 + t_{r+1} \varphi_{r+1} + \cdots +
    t_N \varphi_N,
  \end{aligned}
\end{equation}
so that $\tau = \tau_{\rm two} + \tau_{\rm rest}$ and $t = t_{\rm two}
+ t_{\rm rest}$. Then
\begin{multline}
  \label{eq:bigQCXdivisor}
  \phi_\alpha \QC{\tau} \phi_\beta =
  \sum_{\substack{d \in \NE(\cX): \\
      d = d_1 \pi_\star \beta_1 + \cdots + d_s \pi_\star \beta_s}}
  \sum_{ n \geq 0 }
  \correlator{\phi_{\alpha},\phi_{\beta},\tau_{\rm
      rest},\ldots,\tau_{\rm rest},\phi^\gamma}^\cX_{0,n+3,d} \\
  \times {U_1^{d_1} \cdots U_s^{d_s} \re^{d_1 \tau_1} \cdots \re^{d_s \tau_s}
    \over n!} \, \phi_\gamma
\end{multline}
and
\begin{multline}
  \label{eq:bigQCYdivisor}
  \varphi_\alpha \QC{t} \varphi_\beta =
  \sum_{\substack{d \in \NE(Y): \\
      d = d_1 \beta_1 + \cdots + d_r \beta_r}}
  \sum_{ n \geq 0 }
  \correlator{\varphi_{\alpha},\varphi_{\beta},t_{\rm
      rest},\ldots,t_{\rm rest},\varphi^\gamma}^Y_{0,n+3,d} \\
  \times {Q_1^{d_1} \cdots Q_r^{d_r} \re^{d_1 t_1} \cdots \re^{d_r t_r}
    \over n!} \, \varphi_\gamma.
\end{multline}
Thus in the limit
\[
\begin{aligned}
  \Real \tau_i &\to -\infty, && 1 \leq i \leq s,\\
  \tau_i &\to 0, && \text{$i=0$ and $s<i\leq N$,}
\end{aligned}
\]
the big quantum product $\QC{\tau}$ on $\HorbNovX$ becomes the
Chen--Ruan product, and in the limit
\[
\begin{aligned}
  \Real t_i &\to -\infty, && 1 \leq i \leq r,\\
  t_i &\to 0, && \text{$i=0$ and $r<i\leq N$,}
\end{aligned}
\]
the big quantum product $\QC{t}$ on $\HNovY$ becomes the usual cup
product.  We refer to the points
\begin{align*}
  \tau_i =
  \begin{cases}
    - \infty & 1 \leq i \leq s \\
    0 & \text{otherwise}
  \end{cases}
  && \text{and} &&
  t_i =
  \begin{cases}
    - \infty & 1 \leq i \leq r \\
    0 & \text{otherwise}
  \end{cases}
\end{align*}
as the \emph{large-radius limit points} for $\cX$ and $Y$ respectively.

\subsection*{An Analyticity Assumption and Its Consequences}

The goal of this paper is to describe a relationship between the big
quantum products on $\HorbNovX$ and $\HNovY$.  The first obstacle to
overcome is that the ground rings $\NovX$ and $\NovY$ are in general
not isomorphic: $\NovY$ contains more quantum parameters ($Q_i:1 \leq
i \leq r$) than $\NovX$ does ($U_i: 1 \leq i \leq s$).  We now
describe an analyticity assumption on the big quantum product $\QC{t}$
for $Y$ which allows us to regard $\QC{t}$ as a family of algebra
structures on $H^\bullet(Y;\NovX)$: it allows us to set $Q_i = U_i$,
$1 \leq i \leq s$, and to specialize the extra quantum parameters
$Q_{s+1},\ldots,Q_r$ to $1$.  Roughly speaking, we
assume henceforth that the genus-zero Gromov--Witten potential $F_Y$,
which is a formal power series in the variables $t_0,\ldots,t_N$ and
$Q_1,\ldots,Q_r$, is \emph{convergent in the
  `exceptional variables'} $Q_{s+1},\ldots,Q_r$.

\begin{definition}
  Let $F \in \CC[\![x_0,x_1,x_2,\ldots]\!]$ be a formal power series
  in the variables $x_0,x_1,x_2,\ldots$\phantom{.}  Given distinct
  variables $x_{i_1},\ldots, x_{i_n}$ we can write $F$ uniquely in the
  form
  \[
  F = \sum_{J \subset \NN \setminus \{i_1,\ldots,i_n\}} \,
  \sum_{a:J \to \NN \setminus \{0\}}
  \mathfrak{f}_{J,a}
  \prod_{j \in J} x_j^{a(j)}
  \]
  where each $\mathfrak{f}_{J,a}$ is a formal power series in the
  variables $x_{i_1},\ldots,x_{i_n}$.  Let $D$ be a domain in $\CC^n$
  which contains the origin.  We say that $F$ \emph{depends
    analytically on $x_{i_1},\ldots,x_{i_n}$ in the domain $D$} if
  each $\mathfrak{f}_{J,a}$ is the Taylor expansion at the origin of
  $f_{J,a}(x_{i_1},\ldots,x_{i_n})$ for some analytic function
  $f_{J,a}:D \to \CC$.
\end{definition}

The genus-zero Gromov--Witten potential $F_Y$ is a formal power series
in the variables $t_0,\ldots,t_N$ and
$Q_1,\ldots,Q_r$.  Henceforth, we impose:

\begin{convassum} \label{convassum} There are strictly positive real
  numbers $R_i$, $s<i \leq r$, such that $F_Y$ depends analytically on
  $Q_{s+1},\ldots,Q_r$ in the domain
  \begin{align*}
    |Q_i| < R_i, && s<i \leq r.
  \end{align*}
\end{convassum}

\smallskip

\noindent This assumption holds, for instance, whenever $Y$ is a
compact semi-positive toric manifold.  As we will see, even though the
radii of convergence $R_i$ need not all be greater than $1$, this
assumption will allow us to set $Q_{s+1} = \cdots =
Q_r = 1$.  It follows from \eqref{eq:bigQCYdivisor} that under
Convergence Assumption~\ref{convassum}, $F_Y$ in fact depends
analytically on $t_1,t_2,\ldots,t_r$ and
$Q_{s+1},\ldots,Q_r$ in the domain
\begin{equation}
  \label{eq:firstregion}
  \begin{aligned}
    &|t_i| < \infty & 1 \leq i \leq s \\
    &|Q_i \re^{t_i}|<R_i & s<i \leq r.
  \end{aligned}
\end{equation}
Thus we can write $F_Y$ as
\[
  \sum_{\substack{
      J \subset \{0,r+1,r+2,\ldots,N\} \\
      K \subset \{1,2,\ldots,s\}}}
  \sum_{\substack{
      a:J \to \NN \setminus\{0\} \\
      b:K \to \NN \setminus\{0\}}}
  g_{J,a;K,b}\Big(t_1,\ldots,t_r;Q_{s+1},\ldots,Q_r\Big)
  \prod_{j \in J} t_j^{a(j)} \prod_{k \in K} Q_k^{b(k)},
\]
where $g_{J,a;K,b}$ are analytic functions defined in the domain
\eqref{eq:firstregion}, and then set
\begin{equation}
  \label{eq:subst}
  Q_i =
  \begin{cases}
    U_i & 1 \leq i \leq s \\
    1 & s< i \leq r
  \end{cases}
\end{equation}
obtaining a well-defined power series
\[
F_Y^{\circledast} =
\sum_{\substack{
    J \subset \{0,r+1,r+2,\ldots,N\} \\
    K \subset \{1,2,\ldots,s\}}}
\sum_{\substack{
    a:J \to \NN \setminus\{0\} \\
    b:K \to \NN \setminus\{0\}}}
g_{J,a;K,b}\Big(t_1,\ldots,t_r;1,\ldots,1\Big)
\prod_{j \in J} t_j^{a(j)} \prod_{k \in K} U_k^{b(k)}
\]
in the variables $t_0,t_{r+1},t_{r+2},\ldots,t_N$ and
$U_1,\ldots,U_s$, with coefficients which are analytic
functions of $t_1,\ldots,t_r$ defined in the region
\begin{equation}
  \label{eq:strangeregion}
  \begin{aligned}
    &|t_i| < \infty & 1 \leq i \leq s \\
    &|\re^{t_i}|<R_i & s<i \leq r.
  \end{aligned}
\end{equation}
We can also make the substitution \eqref{eq:subst} in the big quantum
product \eqref{eq:bigQCY}, obtaining a well-defined family of products
$\newQC{t}$ on $H^\bullet(Y;\NovX)$ which depends formally on the
variables $t_0,t_{r+1},t_{r+2},\ldots,t_N$ and analytically on the
variables $t_1,\ldots,t_r$ in the domain \eqref{eq:strangeregion}.
The product $\newQC{t}$ satisfies
\[
\varphi_\alpha \newQC{t} \varphi_\beta =
{\partial^3 F_Y^{\circledast} \over \partial t_\alpha \partial t_\beta
  \partial t_\gamma} \varphi^\gamma
\]
and
\begin{multline}
  \label{eq:newQCYdivisor}
  \varphi_\alpha \newQC{t} \varphi_\beta =
  \sum_{\substack{d \in \NE(Y): \\
      d = d_1 \beta_1 + \cdots + d_r \beta_r}}
  \sum_{ n \geq 0 }
  \correlator{\varphi_{\alpha},\varphi_{\beta},t_{\rm
      rest},\ldots,t_{\rm rest},\varphi^\gamma}^Y_{0,n+3,d} \\
  \times {U_1^{d_1} \cdots U_s^{d_s} \re^{d_1 t_1} \cdots \re^{d_r t_r}
    \over n!} \, \varphi_\gamma
\end{multline}
where $t_{\rm rest}$ is defined in \eqref{eq:tworest}.

We do not impose any convergence assumption on the Gromov--Witten
potential $F_\cX$, which is a formal power series in
$\tau_0,\ldots,\tau_N$ and $U_1,\ldots, U_s$, but
nonetheless it depends analytically on the variables
$\tau_1,\ldots,\tau_s$ in the domain $\CC^s$.  This is clear from
equation \eqref{eq:bigQCXdivisor}.

\section{Givental's Lagrangian Cone}
\label{sec:cone}

The key objects in conjecture~\ref{conj} are certain Lagrangian
submanifold-germs $\cLX$ and $\cLY$.  In this section we define $\cLX$
and $\cLY$ and describe some of their
properties.

\subsection*{A Symplectic Vector Space}

Throughout this section, let $\cZ$ denote either $\cX$ or $Y$.  We
work over the ground ring $\NovZ = \NovX$.  Let
\begin{align*}
  \cHZ &= \HorbNovZ \otimes \CC(\!(z^{-1})\!), \\
  \Omega_\cZ(f,g) &= \Res_{z=0} \big( f(-z), g(z) \big)_{\cZ} \, dz.
\end{align*}
We think of $\cHZ$ as a sort of `symplectic vector space', but defined
over the ring $\Lambda$ rather than over a field.  $\cHZ$ is a free
graded $\Lambda$-module, where $\deg z = 2$, and $\Omega_{\cZ}$ is a
$\NovZ$-linear, $\NovZ$-valued supersymplectic form on $\cHZ$:
\begin{align*}
  \Omega_\cZ(\theta_1 z^k, \theta_2 z^l) =
  (-1)^{a_1 a_2+1}   \Omega_\cZ(\theta_2 z^l, \theta_1 z^k)
  && \text{for $\theta_i \in \HorbiZ{a_i}$.}
\end{align*}
There is a decomposition $\cHZ = \cHZ^+ \oplus \cHZ^-$, where the
subspaces
\begin{align*}
  \cHZ^+ = \HorbNovZ \otimes \CC[z] && \text{and} &&
  \cHZ^- = z^{-1} \HorbNovZ \otimes \CC[\![z^{-1}]\!]
\end{align*}
are Lagrangian.  We can write a general point in $\cHZ$ as
\begin{equation}
  \label{eq:Darboux}
  \sum_{k = 0}^\infty \sum_{a = 0}^N q_{k,a} \Phi_a z^k +
  \sum_{l = 0}^\infty \sum_{b= 0}^N p_{l,b} \Phi^b
  (-z)^{-1-l}
\end{equation}
where $\Phi_a = \phi_a$ and $\Phi^a = \phi^a$ if $\cZ = \cX$, and
$\Phi_a = \varphi_a$ and $\Phi^a = \varphi^a$ if $\cZ = Y$; this
defines $\NovZ$-valued Darboux co-ordinates $\{q_{k,a}, p_{l,b}\}$ on
$\cHZ$, with $q_{k,a}$ dual to $p_{k,a}$.  Set $q_k = \sum_a q_{k,a}
\Phi_a$, so that $\bfq(z) = q_0 + q_1 z + q_2 z^2 + \cdots$ is a
general point in $\cHZ^+$.

\subsection*{The Genus-Zero Descendant Potentials}

We consider now the \emph{genus-zero descendant potentials}
$\cF^0_\cX$ and $\cF^0_Y$, which are generating functions for all
genus-zero Gromov--Witten invariants of $\cX$ and $Y$.  Set $\tau_a =
\tau_{a,\alpha} \phi_\alpha$, $a = 0,1,2,\ldots$\phantom{.}  Then
\begin{align}
  \cF^0_\cX &=
  \sum_{d \in \NE(\cX)}
  \sum_{n \geq 0}
  \sum_{a_1,\ldots,a_n \geq 0}
  \correlator{\tau_{a_1} \psi^{a_1},\tau_{a_2}
    \psi^{a_2},\ldots,\tau_{a_n} \psi^{a_n}}^{\cX}_{0,n,d}
  { U^{d} \over n!} \label{eq:F0X} \\
  &= \notag
  \sum_{d \in \NE(\cX)}
  \sum_{ n \geq 0 }
  \sum_{a_1,\ldots,a_n \geq 0}
  \correlator{\phi_{\epsilon_1} \psi^{a_1},\ldots,
    \phi_{\epsilon_n} \psi^{a_n}}^\cX_{0,n,d}
  {U_1^{d_1} \cdots U_s^{d_s} \tau_{a_1,\epsilon_1} \cdots
  \tau_{a_n,\epsilon_n} \over n!}
\end{align}
where $d = d_1 \pi_\star \beta_1 + \cdots + d_s \pi_\star \beta_s$.
The descendant potential $\cF^0_\cX$ is a formal power series in the
variables $U_1,\ldots,U_s$ and $\tau_{a,\epsilon}$, $0
\leq \epsilon \leq N$, $0 \leq a < \infty$.  We show in the appendix
that $\cF^0_\cX$ in fact depends analytically on
$\tau_{0,1},\ldots,\tau_{0,s}$ in the domain $\CC^s$.  Similarly, set
$t_a = t_{a,\alpha} \varphi_\alpha$, $a = 0,1,2,\ldots$\phantom{.}
Then
\begin{align}
  \cF^0_Y &=
  \sum_{d \in \NE(Y)}
  \sum_{n \geq 0}
  \sum_{a_1,\ldots,a_n \geq 0}
  \correlator{t_{a_1} \psi^{a_1},t_{a_2} \psi^{a_2},\ldots,t_{a_n} \psi^{a_n}}^Y_{0,n,d}
  { Q^d \over n!}  \label{eq:F0Y} \\
  &= \sum_{d \in \NE(Y)}
  \sum_{ n \geq 0 }
  \sum_{a_1,\ldots,a_n \geq 0}
  \correlator{\varphi_{\epsilon_1} \psi^{a_1},\ldots,
    \varphi_{\epsilon_n} \psi^{a_n}}^Y_{0,n,d}
  {Q_1^{d_1} \cdots Q_r^{d_r} t_{a_1,\epsilon_1} \cdots
  t_{a_n,\epsilon_n} \over n!} \notag
\end{align}
where $d = d_1 \beta_1 + \cdots + d_r \beta_r$.  The descendant
potential $\cF^0_Y$ is a formal power series in the variables
$Q_1,\ldots,Q_r$ and $t_{a,\epsilon}$, $0 \leq
\epsilon \leq N$, $0 \leq a < \infty$.  We will show in the appendix
that under convergence assumption \ref{convassum}, $\cF^0_Y$ in fact
depends analytically on $t_{0,1},\ldots,t_{0,r}$ and
$Q_{s+1},\ldots,Q_r$ in the domain
\begin{equation}
  \label{eq:secondregion}
  \begin{aligned}
    &|t_{0,i}| < \infty & 1 \leq i \leq s \\
    &|Q_i \re^{t_{0,i}}|<R_i & s<i \leq r.
  \end{aligned}
\end{equation}
This will allow us, as before, to set $Q_{s+1} = \cdots =
Q_r = 1$: we can write $\cF^0_Y$ as
\begin{multline*}
  \sum_{\substack{
      J \subset \NN \times \{0,1,2,\ldots,N\}: \\
      J \cap \{(0,1),(0,2),\ldots,(0,r)\} = \varnothing}}
  \sum_{K \subset \{1,2,\ldots,s\}} 
  \sum_{\substack{
      a:J \to \NN \setminus\{0\} \\
      b:K \to \NN \setminus\{0\}}}
  g_{J,a;K,b}\Big(t_{0,1},\ldots,t_{0,r};Q_{s+1},\ldots,Q_r\Big) \\
  \times \prod_{(j,e) \in J} t_{j,e}^{a(j,e)} \prod_{k \in K} Q_k^{b(k)}
\end{multline*}
where $g_{J,a;K,b}$ are analytic functions defined in the domain
\eqref{eq:secondregion}, and making the substitution \eqref{eq:subst}
yields a well-defined power series
\begin{multline}
  \label{eq:modifiedF0Y}
  \cF^\circledast_Y = \sum_{\substack{
      J \subset \NN \times \{0,1,2,\ldots,N\}: \\
      J \cap \{(0,1),(0,2),\ldots,(0,r)\} = \varnothing}}
  \sum_{K \subset \{1,2,\ldots,s\}}
  \sum_{\substack{
      a:J \to \NN \setminus\{0\} \\
      b:K \to \NN \setminus\{0\}}} 
  g_{J,a;K,b}\Big(t_{0,1},\ldots,t_{0,r};1,\ldots,1\Big) \\
  \times \prod_{(j,e) \in J} t_{j,e}^{a(j,e)} \prod_{k \in K} U_k^{b(k)}
\end{multline}
in the variables $t_{0,0}$; $t_{0,r+1},t_{0,r+2},\ldots,t_{0,N}$;
$t_{a,\epsilon}$, $0 \leq \epsilon \leq N$, $1 \leq a < \infty$; and
$U_1,\ldots,U_s$, with coefficients which are analytic functions of
$t_{0,1},\ldots,t_{0,r}$ defined in the domain
\begin{equation}
  \label{eq:secondstrangeregion}
  \begin{aligned}
    &|t_{0,i}| < \infty & 1 \leq i \leq s \\
    &|\re^{t_{0,i}}|<R_i & s<i \leq r.
  \end{aligned}
\end{equation}
Thus, exactly as before, Convergence Assumption~\ref{convassum} allows
us to work over the Novikov ring $\Lambda = \Lambda_\cX$ for $\cX$,
even when we are thinking about Gromov--Witten invariants of $Y$.

\subsection*{The Definition of $\cLX$ and $\cLY$}

We regard the genus-zero descendant potential $\cF^0_\cX$ as the germ
of a function on $\cHX^+$ via the identification
\begin{equation}
  \label{eq:dilatonshiftX}
  q_{k,\alpha} =
  \begin{cases}
    \tau_{1,0} - 1 & (k,\alpha) = (1,0) \\
    \tau_{k,\alpha} & \text{otherwise,}
  \end{cases}
\end{equation}
which we abbreviate as $\bfq(z) = \btau(z) - z$.  We regard
$\cF^\circledast_Y$ as the germ of a function on $\cHY^+$ via the
identification
\begin{equation}
  \label{eq:dilatonshiftY}
  q_{k,\alpha} =
  \begin{cases}
    t_{1,0} - 1 & (k,\alpha) = (1,0) \\
    t_{k,\alpha} & \text{otherwise,}
  \end{cases}
\end{equation}
which we abbreviate as $\bfq(z) = \bt(z) - z$.  The identifications
\eqref{eq:dilatonshiftX} and \eqref{eq:dilatonshiftY} are examples of
the \emph{dilaton shift}; this is discussed further in \cite{Coates}.
Let $\cF_\cZ = \cF^0_\cX$ if $\cZ = \cX$ and $\cF_\cZ =
\cF^\circledast_Y$ if $\cZ = Y$.  We define $\cLZ$ by the equations
\begin{align}
  \label{eq:defofLZ}
  p_{k,\alpha} =\parfrac{\cF_\cZ}{q_{k,\alpha}} &&
  \begin{aligned}
    0 &\leq k < \infty, \\
    0 &\leq \alpha \leq N.
  \end{aligned}
\end{align}
As $\cF_\cZ$ is the germ of a function on $\cHZ^+$ (depending
analytically on some variables and formally on other variables),
$\cLZ$ is the germ of a Lagrangian submanifold of $\cHZ$.

\begin{rem}
  The polarization $\cHZ = \cHZ^+ \oplus \cHZ^-$ identifies $\cHZ^-$
  with the $\Lambda$-module $\(\cHZ^+\)^\star := \Hom(\cHZ^+,\Lambda)$
  dual to $\cHZ^+$, and hence identifies $\cHZ$ with the cotangent
  bundle $T^\star \cHZ^+ := \cHZ^+ \oplus \(\cHZ^+\)^\star$.  Under
  this identification, $\cLZ$ becomes the graph of the differential of
  $\cF_\cZ$.
\end{rem}

The Gromov--Witten invariants which participate in the definition of
$\cLZ$ satisfy a large number of identities: the String Equation, the
Dilaton Equation, and the Topological Recursion Relations.  These
identities place very strong constraints on the geometry of $\cLZ$:

\begin{theorem}[\cite{Coates--Givental:QRRLS,Givental:symplectic,Tseng}]
  \label{thm:cone} $\cLZ$ is the germ of a Lagrangian cone with vertex
  at the origin such that each tangent space $T$ to $\cLZ$ is tangent
  to the cone exactly along $zT$.  In other words:
  \begin{enumerate}
  \item if $T$ is a tangent space to $\cLZ$ then $zT \subset T$;
  \item if $T = T_x \cLZ$ then the germ at $x$ of the linear subspace
    $zT$ is contained in $\cLZ$;
  \item if $T$ is a tangent space to $\cLZ$ and $x \in \cLZ$ then $T_x
    \cLZ = T$ if and only if $x \in zT$.
  \end{enumerate}
\end{theorem}

In particular, theorem~\ref{thm:cone} implies that each tangent space
$T$ to $\cLZ$ is closed under multiplication by elements of $\CC[z]$
(because $zT \subset T$), and that $\cLZ$ is the union, over all
tangent spaces $T$ to $\cLZ$, of the infinite-dimensional linear
subspace-germs $zT \cap \cLZ$.  It is the germ of a `ruled cone'.
Note that as $\cLZ$ is the germ of a submanifold of $\cHZ$, it makes
sense to analytically continue $\cLZ$.

\section{The Crepant Resolution Conjecture}
\label{sec:conj}

We are now in a position to make our conjecture.

\begin{conj}[Coates--Corti--Iritani--Tseng; Ruan]
  \label{conj} There is a degree-preserving $\CC(\!(z^{-1})\!)$-linear
  symplectic isomorphism $\U:\cHX \to \cHY$ and a choice of analytic
  continuations of $\cLX$ and $\cLY$ such that $\U\(\cLX\) = \cLY$.
  Furthermore, $\U$ satisfies:
  \begin{itemize}
  \item[(a)] $\U(\fun_\cX) = \fun_Y + O(z^{-1})$;
  \item[(b)] $\U \circ \(\rho \CR\) = \({\pi^\star \rho} \, \cup\)
    \circ \U$ for every untwisted degree-two class $\rho \in
    H^2(\cX;\CC)$;
  \item[(c)] $\U\(\cHX^+ \)\oplus \cHY^- = \cHY$;
  \item[(d)] the matrix entries of $\U$ with respect to the bases
    $\{\phi_\alpha\}$ and $\{\varphi_\beta\}$, which \emph{a priori}
    are elements of $\Nov(\!(z^{-1})\!)$, in fact lie in
    $\CC(\!(z^{-1})\!)$.
  \end{itemize}
\end{conj}

\begin{rem}
  This conjecture emerged in two different contexts during the ``New
  Topological Structures in Physics'' program at the Mathematical
  Sciences Research Institute, Berkeley, in the spring of
  2006. Conversations between the authors led to the idea that the
  relationship between the quantum cohomology of $\cX$ and $Y$ should
  be expressed as the assertion that $\U(\cLX) = \cLY$ for some
  $\CC(\!(z^{-1})\!)$-linear symplectic isomorphism $\U$. At the same
  time, guided by mirror symmetry, Hiroshi Iritani found such a
  symplectic transformation in toric examples (as a part of a project
  \cite{CCIT:crepant1} with Coates, Corti, and Tseng). Condition (c)
  here is a stronger version of the condition (c) given in
  \cite{CCIT:crepant1}*{\S5}. We will need this stronger version for
  the Cohomological Crepant Resolution Conjecture below.
\end{rem}

\begin{rem}
  Variants of conjecture~\ref{conj} apply to the $G$-equivariant
  quantum cohomology of $G$-equivariant crepant resolutions, and to
  crepant resolutions of certain non-compact orbifolds (\emph{c.f.}
  \cite{Bryan--Graber}).  We leave the necessary modifications to the
  reader.
\end{rem}

\subsection*{What Do The Conditions Mean?}

Without condition (a) any non-zero scalar multiple of $\U$ would also
satisfy the conjecture, because $\cLX$ and $\cLY$ are germs of cones.
The fact that $\U$ is degree-preserving forces $\U(\fun_\cX) = \lambda
\fun_Y + O(z^{-1})$ for some scalar $\lambda$, and so condition (a)
just fixes this overall scalar multiple.

Condition (b) is a compatibility of monodromy.  The A-model connection
--- a system of differential equations associated to the small quantum
cohomology of $Y$ \cite{Cox--Katz}*{\S8.5} --- is regular singular
along the normal-crossing divisor $Q_1 Q_2 \cdots Q_r = 0$, and the
log-monodromy around $Q_i = 0$ is given by cup product with
$\varphi_i$; a similar statement holds for $\cX$.  Condition (b)
asserts that $\U$ matches up these monodromies.

Condition (c) ensures that both the quantum cohomology of $\cX$ and
the analytic continuation of the quantum cohomology of $Y$ make sense
near the large-radius limit point for $\cX$.  This is explained in
detail in Remark~\ref{rem:twoproducts} below.

Condition (d) says that $\U$ is `independent of Novikov variables'.

\section{Basic Properties of the Transformation $\U$}
\label{sec:properties}

Before we explore the implications of conjecture~\ref{conj}, we list
various basic properties of the transformation $\U$.  As we have
chosen homogeneous bases for $\HorbX$ and $H^\bullet(Y;\CC)$ and as
$\U$ is grading-preserving, we can represent the transformation $\U$
by an $(N+1) \times (N+1)$ matrix, each entry of which is a Laurent
monomial in $z$ of fixed degree.  The matrix entries are independent
of Novikov variables, so each entry is the product of a complex number
and a fixed power of $z$.  $\U$ is therefore a Laurent polynomial in
$z$.  For example, if $\cX = \PP(1,1,1,3)$, $Y = \FF_3$, and we choose
bases as in \cite{CCIT:crepant1}, then
\[
\U =
\begin{pmatrix}
1 & 0 &0 & 0&0 &0  \\
0 & 1 &0 &0 &0 &0  \\
0 & 0 &1 &0 & 0 &0  \\
0 & 0 &0 &0 & -\frac{2\sqrt{3}\pi}{3\Gamma(\frac{1}{3})^3}z &
\frac{2\sqrt{3}\pi}{3\Gamma(\frac{2}{3})^3}  \\
-\frac{\pi^2}{3}z^{-2} & 0 &0 &0 &\frac{2\pi^2}{3\Gamma(\frac{1}{3})^3} &
\frac{2\pi^2}{3\Gamma(\frac{2}{3})^3}z^{-1} \\
-{8\zeta(3)} z^{-3} &0 &0 & 1 &
-\frac{2\sqrt{3}\pi^3}{9\Gamma(\frac{1}{3})^3} z^{-1} &
\frac{2\sqrt{3}\pi^3}{9\Gamma(\frac{2}{3})^3} z^{-2}
\end{pmatrix}.
\]
This illustrates the fact that even if the Gromov--Witten invariants
of $\cX$ and $Y$ are defined over $\QQ$, the transformation $\U$ may
only be defined over $\CC$.  Note that some of the matrix entries here
are `highly transcendental'.

\begin{lemma} \label{lem:lowdegree} Suppose that $\omega_i \in
  H^i_{\scriptscriptstyle \text{\rm CR}}(\cX;\CC)$.  Then:
  \begin{itemize}
  \item[(a)] $\U(\omega_{2r}) = z^r \rho_0 + O(z^{r-1})$ for some
    $\rho_0 \in H^0(Y;\CC)$, and if $\rho_0 \ne 0$ then $r=0$;
  \item[(b)] $\U(\omega_{2r+1}) = z^r \rho_1 + O(z^{r-1})$ for some
    $\rho_1 \in H^1(Y;\CC)$, and if $\rho_1 \ne 0$ then $r=0$.
  \item[(c)] $\U(\omega_{2r+2}) = z^r \rho_2 + O(z^{r-1})$ for some
    $\rho_2 \in H^2(Y;\CC)$, and if $\rho_2 \not \in \ker \pi_!$ then
    $r=0$.
  \end{itemize}
\end{lemma}

\begin{proof}
  (a) As $\U$ is grading-preserving, $\U(\omega_{2r}) = z^r \lambda
  \fun_Y + O(z^{r-1})$ for some $\lambda \in \CC$.  Write $D =
  \dim_{\CC}(\cX)$ and suppose that $\lambda \ne 0$.  Then, as $\cX$
  is K\"ahler and as the map $\pi^\star:H^\bullet(\cX;\CC) \to \HY$ is
  injective, there exists $\omega \in H^2(\cX;\CC)$
  such that $(\pi^\star \omega)^D \in H^{2D}(Y;\CC)$ is non-zero.  We
  have
  \begin{align*}
    \U\big(\overbrace{\omega \CR
    \cdots \CR
    \omega}^D \CR
    \omega_{2r} \big) &= z^r \lambda (\pi^\star \omega)^D + O(z^{r-1})
    \\
    & \ne 0,
  \end{align*}
  and hence $(\omega \CR)^D \CR \omega_{2r} \ne 0$.  For degree
  reasons, $r$ must be zero.

  (b) As $\U$ is grading-preserving, $\U(\omega_{2r+1}) = z^r \rho_1 +
  O(z^{r-1})$ for some $\rho_1 \in H^1(Y;\CC)$.  As
  $\pi^\star:H^1(\cX;\CC) \to H^1(Y;\CC)$ is an isomorphism, we have
  $\rho_1 = \pi^\star \theta_1$ for some $\theta_1 \in H^1(\cX;\CC)$.
  Suppose that $\rho_1 \ne 0$.  By Hard Lefschetz for
  $H^\bullet(\cX;\CC)$ (ordinary cohomology not Chen--Ruan
  cohomology), there exists $\omega \in H^2(\cX;\CC)$ such that
  $\omega^{D-1} \theta_1 \in H^{2D-1}(\cX;\CC)$ is non-zero.
  Injectivity of $\pi^\star$ gives $(\pi^\star \omega)^{D-1} \rho_1
  \ne 0$, and so
  \begin{align*}
    \U\big(\overbrace{\omega \CR
    \cdots \CR
    \omega}^{D-1} \CR
    \omega_{2r+1} \big) &= z^r (\pi^\star \omega)^{D-1} \rho_1 + O(z^{r-1})
    \\ & \ne 0.
  \end{align*}
  As before, this forces $r=0$.

  (c) As $\U$ is grading-preserving, $\U(\omega_{2r+2}) = z^r \rho_2 +
  O(z^{r-1})$ for some $\rho_2 \in H^2(Y;\CC)$.  Suppose that $\rho_2
  \not \in \ker \pi_!$.  Then there exist $\omega, \omega' \in
  H^2(\cX;\CC)$ such that $\int_\cX \pi_! \rho_2 \cup \omega^{D-2}
  \cup \omega' \ne 0$; here we used the non-degeneracy of the
  Poincar\'e pairing and Hard Lefschetz for $H^\bullet(\cX;\CC)$.
  Thus $\int_Y \rho_2 \cup \pi^\star \omega^{D-2} \cup \pi^\star
  \omega' \ne 0$, and so $\U(\omega_{2r+2}) \cup \pi^\star
  \omega^{D-2} \cup \pi^\star \omega' \ne 0$.  But
  \[
  \U(\omega_{2r+2}) \cup \pi^\star
  \omega^{D-2} \cup \pi^\star \omega' =
  \U\big(\overbrace{\omega \CR
    \cdots \CR
    \omega}^{D-2} \CR \omega' \CR
  \omega_{2r+2} \big)
  \]
  and as this is non-zero we must, for degree reasons, have $r=0$.
\end{proof}

\begin{lemma} \label{lem:U0} Suppose that $\U$ sends $\cHX^-$ to
  $\cHY^-$, so that
  \[
  \U = U_0 + U_1 z^{-1} + \cdots + U_k z^{-k}
  \]
  for some non-negative integer $k$ and some linear maps $U_i : \HorbX
  \to H^\bullet(Y;\CC)$.  Then:
  \begin{itemize}
  \item[(i)] $U_0$ is grading-preserving;
  \item[(ii)] $U_0$ maps $\fun_\cX$ to $\fun_\cY$;
  \item[(iii)] $U_0$ maps $\rho \in H^2(\cX;\CC)$ to $\pi^\star \rho
    \in H^2(Y;\CC)$;
  \item[(iv)] $U_0$ identifies the orbifold Poincar\'e pairing on
    $\HorbX$ with the Poincar\'e pairing on $H^\bullet(Y;\CC)$.
  \end{itemize}
\end{lemma}
\begin{proof}
  (i) $\U$ is grading-preserving.  (ii) conjecture~\ref{conj}(a).
  (iii) conjecture~\ref{conj}(b).  (iv)~$\U$ is a symplectic
  isomorphism.
\end{proof}

\section{From Givental's Cone to Quantum Cohomology}
\label{sec:conetoQC}

Since $\cLX$ encodes all genus-zero Gromov--Witten invariants of
$\cX$, it implicitly encodes the big quantum product for $\cX$.  In
the same way, $\cLY$ encodes the big quantum product for $Y$.  In this
section we describe how to determine the quantum products from $\cLX$
and $\cLY$, using the geometric structure described in theorem
\ref{thm:cone}.  The big quantum products can be regarded in three
different ways:
\begin{enumerate}
\item as \emph{families of Frobenius algebras}, since
  \begin{align*}
    \Big(u \QC{\tau} v, w \Big)_\cX = \Big(u, v \QC{\tau} w \Big)_\cX
    && \text{and} &&
    \Big(u' \newQC{t} v', w' \Big)_Y = \Big(u', v' \newQC{t} w' \Big)_Y
  \end{align*}
  for all $u,v,w \in \HorbX$ and $u',v',w' \in \HYNovX$.
\item as \emph{F-manifolds}.  An F-manifold is, roughly speaking, a
  Frobenius manifold without a pairing.  It is a manifold equipped
  with a supercommutative associative multiplication on the tangent
  sheaf and a global unit vector field such that the multiplication
  $\circ$ satisfies
  \begin{equation}
    \label{eq:Fintegrability}
    \Lie_{X \circ Y} (\circ) = X \circ \Lie_Y(\circ) + Y
    \circ\Lie_X(\circ)
  \end{equation}
  for any two local vector fields $X$ and $Y$.  F-manifolds are
  studied in \cite{Hertling--Manin,Hertling}.
\item as \emph{Frobenius manifolds}.  A Frobenius manifold is a
  manifold $M$ equipped with the structure of a unital Frobenius
  algebra on each tangent space $T_x M$ such that the associated
  metric on $TM$ is flat, the identity vector field is flat, and
  certain integrability conditions hold (these include the celebrated
  WDVV~equations).  Frobenius manifolds are studied in \cite{Dubrovin,
    Manin}.
\end{enumerate}
Once again, write $\cZ$ for either $\cX$ or $Y$.  In this section, we
will see how to pass from $\cLZ$ to:
\begin{enumerate}
\item a family of Frobenius algebras.  This family is \emph{intrinsic}
  to $\cLZ$ in that it depends only on the symplectic space $\cHZ$ and
  on $\cLZ \subset \cHZ$ satisfying the conclusions of
  theorem~\ref{thm:cone}; it is independent of the
    polarization $\cHZ = \cHZ^+ \oplus \cHZ^-$ used to define $\cLZ$.
\item an F-manifold.  This depends, up to isomorphism, only on $\cHZ$,
  $\cLZ$, and a choice of point on $\cLZ$.
\item a Frobenius manifold.  This depends on $\cHZ$, $\cLZ$, a point
  $x$ of $\cLZ$, and a choice of \emph{opposite subspace} $\cHZ^\opp
  \subset \cHZ$.  Choosing $x \in \cLZ$ appropriately and taking
  $\cHZ^\opp = \cHZ^-$ gives the Frobenius manifold corresponding to
  the quantum cohomology of $\cZ$; we explain this in
  \S\ref{sec:conetoQC}(d--e) below.
\end{enumerate}

Once we understand points 1--3 here, we will see how
conjecture~\ref{conj} implies previous versions of the Crepant
Resolution Conjecture.  If the symplectic transformation $\U$ maps
$\cHX^-$ to $\cHY^-$ then we obtain from point 3 above an isomorphism
between the Frobenius manifolds defined by the quantum cohomologies of
$\cX$ and $Y$.  The Hard Lefschetz condition postulated by
Bryan--Graber in \cite{Bryan--Graber} implies that $\U(\cHX^-) =
\cHY^-$ (this is theorem~5.4 in \cite{CCIT:crepant1}), and so
conjecture~\ref{conj} implies the Bryan--Graber version of the Crepant
Resolution Conjecture.  This is discussed further in \S\ref{sec:BG}.
In general $\U$ will not map $\cHX^-$ to $\cHY^-$ --- in other words,
some of the matrix entries of $\U$ will contain strictly positive
powers of $z$ --- and so $\U$ will \emph{not} induce an isomorphism
between quantum cohomology Frobenius manifolds.  From point 2 above we
still obtain, however, an isomorphism of F-manifolds.  If $\cX$ is
semi-positive then more is true, and we obtain an isomorphism between
the \emph{small} quantum cohomology algebras of $\cX$ and $Y$ which
preserves the Poincar\'e pairings.  This is something very like Ruan's
original Crepant Resolution Conjecture, and we discuss it further in
\S\ref{sec:Ruan}.  Finally, without any additional assumptions on
$\cX$ or $Y$ (no Hard Lefschetz, no semi-positivity) we obtain from
point 1 above something very like the Cohomological Crepant Resolution
Conjecture; we discuss this in \S\ref{sec:CCRC}.

The ideas presented in this section are due to Barannikov and
Givental.  Closely-related discussions can be found in
\citelist{\cite{Barannikov}\cite{Givental:symplectic}\cite{CCIT:crepant1}}.

\subsection{From Givental's Cone to a Family of Frobenius Algebras}

Given $\cLZ \subset \cHZ$ satisfying the conclusions of
theorem~\ref{thm:cone} and a point $x \in \cLZ$, the quotient $T_x/z
T_x$, where $T_x = T_x \cLZ$, inherits the structure of a Frobenius
algebra as follows.  The $\NovZ$-bilinear form
\begin{align*}
  T_x &\otimes T_x \longrightarrow \NovZ \\
  v & \otimes w \longmapsto \Omega(v, z^{-1} w)
\end{align*}
is symmetric and vanishes whenever $v$ or $w$ lies in $zT_x$, so it
descends to give a symmetric bilinear form
\begin{equation}
  \label{eq:TzTpairing}
  g(v + z T_x, w + z T_x) = \Omega(v,z^{-1} w)
\end{equation}
on $T_x/z T_x$.  This form is non-degenerate as $T_x$ is maximal
isotropic.  Choosing a Lagrangian subspace $V$ such that $\cHZ = T_x
\oplus V$ --- one could, for instance, take $V =\cHZ^-$ --- identifies
$V$ with $T_x^\star := \Hom(T_x,\NovZ)$ and $\cHZ$ with the cotangent
bundle $T_x \oplus T_x^\star$.  As $\cLZ$ is Lagrangian, there is the
germ of a function $\phi:T_x \to \NovZ$ such that $\phi(x) = 0$ and
that $\cLZ$ coincides, in a formal neighbourhood of $x$, with the
graph of the differential of $\phi$.  The third derivative $d^3 \phi
|_x$ defines a cubic tensor on $T_x$; it is easy to see that this is
independent of the choice of $V$.  Theorem~\ref{thm:cone} implies that
$\phi$ vanishes identically along the germ of $z T_x \subset T_x$, and
as $d^3 \phi|_x(u,v,w)$ vanishes whenever one of $u,v,w$ lies in
$zT_x$ we obtain a cubic tensor $c$ on $T_x/zT_x$:
\[
c\big(u + z T_x,v + z T_x,w + z T_x\big) = d^3 \phi|_x(u,v,w).
\]
The tensors $c$ and $g$ together define a supercommutative product
$\star$ on $T_x/zT_x$, via
\[
g\Big((u + z T_x) \star(v + z T_x),w + z T_x\Big) = c\big(u + z T_x,v + z
T_x,w + z T_x\big).
\]
The product $\star$ automatically has the Frobenius property with
respect to $g$.  We will see in the next section that it is
associative and unital; the unit depends upon the point $x \in \cLZ$,
so even if the tangent spaces $T_{x_1} = T_{x_1} \cLZ$ and $T_{x_2} =
T_{x_2} \cLZ$ coincide, the algebra structures on $T_{x_1}/z T_{x_1}$
and $T_{x_2}/z T_{x_2}$ will in general differ.  Thus we have obtained
from $\cLZ$ a vector bundle
\[
T\cLZ / z T \cLZ \to \cLZ
\]
such that the fibers of this vector bundle form a family of Frobenius
algebras.

\begin{rem}
  The construction here resembles the construction of the Yukawa
  coupling in the B-model of topological string theory associated to a
  Calabi--Yau $3$-fold (see \cite{Cox--Katz} and \emph{e.g.}
  \cite{Givental:homological}*{\S6}).  This is not an accident.  The
  tangent spaces $T$ to $\cLZ$ form a \emph{variation of semi-infinite
    Hodge structure} in the sense of Barannikov \cite{Barannikov}, and
  part of the power of Barannikov's theory is that it can describe
  A-model phenomena (like quantum cohomology) and B-model phenomena in
  the same language.
\end{rem}

\begin{rem}
  If we take $\cX$ to be a manifold, $\cZ = \cX$, $V = \cHX^-$, and
  the point $x \in \cLX$ to be $J_\cX(\tau,-z)$, defined in
  \S\ref{sec:conetoQC}(d) below, then the function-germ $\phi$
  described above is Givental's \emph{genus-zero ancestor potential}
  $\bar{\mathcal{F}}_\tau^0$ of $\cX$
  \cite{Givental:quantization}*{\S5}.
\end{rem}

\subsection{From Givental's Cone to an F-Manifold}

Given $\cLZ \subset \cHZ$ satisfying the conclusions of
theorem~\ref{thm:cone} and a point $x \in \cLZ$, we construct an
F-manifold as follows.  Let $T_x = T_x \cLZ$ and choose a Lagrangian
subspace $V \subset \cHZ$ such that $\cHZ = T_x \oplus V$.  Let $M =
T_x \cap z V$.  Our F-manifold will be based on a formal neighbourhood
of the origin in $M$.

As $\cLZ$ is the graph of a germ of a map from $T_x$ to $V$, there is
a unique germ of a function $K:M \to \cHZ$ such that $K(t) \in \cLZ$
and $K(t) = x + t + v(t)$ for some $v(t) \in V$.  Choose a basis
$e_0,\ldots,e_N$ for $M$ and denote the corresponding linear
co-ordinates on $M$ by $t_a$, $0 \leq a \leq N$.

\begin{proposition} \label{pro:TzTbasis}
  For $t$ in a formal neighbourhood of the origin in $M$, the elements
  \begin{align}
    \label{eq:TzTbasis}
    {\partial K \over \partial t_a}(t) + z T_{K(t)}, &&
    a = 0,1,\ldots,N,
  \end{align}
  form a basis for $T_{K(t)}/z T_{K(t)}$.
\end{proposition}

\begin{proof}
  It suffices to prove this at $t=0$.  But $K(0) = x$ and, since $T_x$
  is tangent to $\cLZ$ at $x$, ${\partial K \over \partial t_a}(0)$
  has no component along $V$: ${\partial K \over \partial t_a}(0) =
  e_a$.  So we need to show that
  \begin{align*}
    e_a + z T_x    && a = 0,1,\ldots,N,
  \end{align*}
  form a basis for $T_x/zT_x$.  This holds because $\cHZ = z T_x
  \oplus z V$, and so the projection $M = T_x \cap z V \to T_x/z T_x$ is
  an isomorphism.
\end{proof}

Thus for $t$ in a formal neighbourhood $M_0$ of the origin in $M$, the
map $DK|_t:T_t M \to T_{K(t)}/z T_{K(t)}$ is an isomorphism.  Pulling
back the Frobenius algebra structure defined in the previous section
via the map $DK$ gives a pairing
\[
g_{\alpha \beta}(t) =
\Omega\bigg(
{\partial K \over \partial t_\alpha}(t),
z^{-1} {\partial K \over \partial t_\beta}(t)
\bigg)
\]
and a symmetric $3$-tensor
\[
c_{\alpha \beta \gamma}(t) =
\Omega\bigg(
{\partial^2 K \over \partial t_\beta \partial t_\gamma}(t),
{\partial K \over \partial t_\alpha}(t)
\bigg)
\]
on $T_t M_0$.  Denote the induced product on $T_t M_0$ by $\circ_t$:
\[
e_\alpha \circ_t e_\beta = c_{\alpha \beta}^{\phantom{\alpha \beta}
  \gamma}(t) e_\gamma
\]
where $c_{\alpha \beta \gamma}(t) = c_{\alpha \beta}^{\phantom{\alpha
    \beta} \epsilon}(t) g_{\epsilon \gamma}(t)$.

\begin{proposition} \ \label{pro:Fprodisassociative}
  \begin{itemize}
  \item[(a)] $\nabla_{u \circ_t v} K(t) + z T_{K(t)} = -z \nabla_u
    \nabla_v K(t) + z T_{K(t)}$, where $\nabla_u = u^\alpha {\partial
      \over \partial t_\alpha}$ denotes the directional derivative
    along $u = u^\alpha e_\alpha$.
  \item[(b)] The tensor $c_{\alpha \beta}^{\phantom{\alpha \beta}
      \epsilon}(t) c_{\epsilon \gamma \delta}(t)$ is symmetric in
    $\alpha$, $\beta$, $\gamma$, $\delta$.
  \item[(c)] The product $\circ_t$ is associative.
  \end{itemize}
\end{proposition}

\begin{proof}
  As $c_{\gamma \beta \alpha}(t) = c_{\alpha \beta}^{\phantom{\alpha
      \beta} \epsilon}(t) g_{\gamma \epsilon}(t)$, we have
  \[
  \Omega\bigg(
  {\partial^2 K \over \partial t_\beta \partial t_\alpha}(t),
  {\partial K \over \partial t_\gamma}(t)
  \bigg)
  =
  \Omega\bigg(
  {\partial K \over \partial t_\gamma}(t),
  z^{-1} c_{\alpha \beta}^{\phantom{\alpha
      \beta} \epsilon}(t)
  {\partial K \over \partial t_\epsilon}(t)
  \bigg).
  \]
  The pairing \eqref{eq:TzTpairing} is non-degenerate, and
  \eqref{eq:TzTbasis} is a basis for $T_{K(t)}/z T_{K(t)}$, so
  \begin{equation}
    \label{eq:Fproduct}
    - z {\partial^2 K \over \partial t_\alpha \partial t_\beta}(t)
    + z T_{K(t)}
    =
    c_{\alpha \beta}^{\phantom{\alpha \beta} \epsilon}(t)
    {\partial K \over \partial t_\epsilon}(t)
    + z T_{K(t)}.
  \end{equation}
  This proves (a).  Theorem~\ref{thm:cone} implies that if $y(t) \in
  T_{K(t)}$ then $z {\partial y \over \partial t_a}(t) \in T_{K(t)}$
  too, so differentiating \eqref{eq:Fproduct} yields
  \begin{align*}
    z^2 {\partial^3 K \over \partial t_\alpha \partial t_\beta
      \partial t_\gamma}(t)
    + z T_{K(t)}
    &=
    - c_{\alpha \beta}^{\phantom{\alpha \beta} \epsilon}(t)
    z {\partial^2 K \over \partial t_\epsilon \partial t_\gamma}(t)
    + z T_{K(t)} \\
    &=
    c_{\alpha \beta}^{\phantom{\alpha \beta} \epsilon}(t)
    c_{\epsilon \gamma}^{\phantom{\epsilon \gamma} \delta}(t)
    {\partial K \over \partial t_\delta}(t)
    + z T_{K(t)}.
  \end{align*}
  Thus $c_{\alpha \beta}^{\phantom{\alpha \beta} \epsilon}(t)
  c_{\epsilon \gamma}^{\phantom{\epsilon \gamma} \delta}(t)$ is
  symmetric in $\alpha$, $\beta$, $\gamma$.  As $c_{\epsilon \gamma
    \delta}(t)$ is symmetric as well, part (b) follows.  Part (c) is
  an immediate consequence of part (b).
\end{proof}

So far, we have constructed a family of supercommutative associative
products on the fibers of $TM_0$, depending on $\cLZ \subset \cHZ$, a
point $x \in \cLZ$, and a Lagrangian subspace $V$.  To prove that this
makes $M_0$ into an F-manifold we need to show that the algebras
$(T_tM_0,\circ_t)$ are unital and that the integrability condition
\eqref{eq:Fintegrability} holds.  After that we will show that, up to
isomorphism, the F-manifold we have constructed is independent of the
choice of Lagrangian subspace $V$.

Define a vector field $e$ on $M_0$ by
\[
\nabla_{e(t)} K(t) + z T_{K(t)} = - z^{-1} K(t) + z T_{K(t)}.
\]
This makes sense, as $z^{-1} K(t) \in T_{K(t)}$ by theorem~\ref{thm:cone}.

\begin{proposition} \label{pro:identity}
  $e(t)$ is the identity element in the algebra $(T_t(M_0),\circ_t)$.
\end{proposition}

\begin{proof}
  Let $v$ be any vector field on $M_0$.  Then
  \begin{align*}
    \nabla_{e(t) \circ_t v(t)} K(t) + z T_{K(t)} &=
    - z \nabla_v(t) \nabla_{e(t)} K(t) + z T_{K(t)} \\
    &= \nabla_{v(t)} K(t) + z T_{K(t)}
  \end{align*}
  and so $e(t) \circ_t v(t) = v(t)$.
\end{proof}

\begin{cor}
  The product on $T_x / z T_x$ constructed in \S\ref{sec:conetoQC}(a)
  is associative and unital.
\end{cor}

\begin{proof}
  Set $t=0$ in propositions~\ref{pro:Fprodisassociative}(c)
  and~\ref{pro:identity}.
\end{proof}

\begin{proposition}
  The triple $(M_0,\circ,e)$ is an F-manifold.
\end{proposition}

\begin{proof}
  It remains only to establish the integrability condition
  \eqref{eq:Fintegrability}, and for this the argument of
  \cite{Hertling--Manin}*{\S2} applies.  The essential ingredients
  there are proposition~\ref{pro:Fprodisassociative}(b) and that the
  quantity ${\partial \over \partial t_\delta} c_{\alpha \beta
    \gamma}(t)$ is symmetric in $\alpha, \beta,\gamma,\delta$: the
  latter assertion holds here as ${\partial \over \partial t_\delta}
  c_{\alpha \beta \gamma}(t)$ is the fourth derivative of a function
  $\phi:M_0 \to \Nov$.
\end{proof}

\begin{proposition}
  Suppose that $\cLZ \subset \cHZ$ satisfies the conclusions of
  theorem~\ref{thm:cone}, that $x \in \cLZ$, that $T_x = T_x \cLZ$,
  and that $V, V' \subset \cHZ$ are Lagrangian subspaces such that
  $T_x \oplus V = T_x \oplus V' = \cHZ$.  Let $(M_0,\circ,e)$ and
  $(M_0',\circ',e')$ be the corresponding F-manifolds, and
  \begin{align*}
    K:M_0 \to \cHZ, &&
    K':M_0' \to \cHZ,
  \end{align*}
  be the corresponding functions (constructed just above
  proposition~\ref{pro:TzTbasis}).  Then there is a unique map $f:M_0
  \to M_0'$ and a unique section $w$ of $K^\star T \cLZ$ (\emph{i.e.}
  a unique choice of $w(t) \in T_{K(t)} \cLZ$) such that
  \begin{align}
    \label{eq:KK'}
    K'(f(t)) = K(t) + z w(t), && \text{for all $t \in M_0$}.
  \end{align}
  The map $f$ gives an isomorphism of F-manifolds between
  $(M_0,\circ,e)$ and $(M_0',\circ',e')$.
\end{proposition}

\begin{proof}
  Let $\pi':\cHZ \to T_x$ denote the projection along $V'$, and for $y
  \in \cLZ$ write $T_y = T_y \cLZ$.  Recall that $M_0$, $M_0'$ are
  formal neighbourhoods of the origins in
  \begin{align*}
    M = T_x \cap z V, && M' = T_x \cap z V'
  \end{align*}
  respectively, and that $K(t)$, $K'(t')$ are the unique elements of
  $\cLZ$ of the form
  \begin{align*}
    K(t) = x + t + v(t), &&
    K'(t') = x' + t' + v'(t'),
  \end{align*}
  where $t \in M_0$, $v(t) \in V$, $t' \in M_0'$, and $v'(t') \in V'$.

  We begin by showing that, for all $t \in M_0$, $T_x = \pi'\big(z
  T_{K(t)}\big) \oplus M'$.  It suffices to prove this at $t=0$, and
  since $K(0) = x$ we need to show that $T_x = z T_x \oplus M'$.
  This follows from the fact that the projection $M' \to T_x / z T_x$
  is an isomorphism (\emph{c.f.} the proof of
  proposition~\ref{pro:TzTbasis}).  So $T_x = \pi'\big(z T_{K(t)}\big)
  \oplus M'$ for all $t \in M_0$.

  There is therefore a unique element $w(t) \in T_{K(t)}$ such that
  \[
  \pi' \big[ K(t) + z w(t) \big] \in x + M'.
  \]
  Theorem~\ref{thm:cone} implies that $K(t) + z w(t) \in \cLZ$, and so
  setting
  \[
  f(t) = \pi' \big[ K(t) + z w(t) \big] - x
  \]
  gives a map $f:M_0 \to M_0'$ such that
  \[
  K'(f(t)) = K(t) + z w(t).
  \]
  This shows existence of a map $f:M_0 \to M_0'$ and a section $w$ of
  $K^\star T \cLZ$ satisfying \eqref{eq:KK'}; uniqueness is clear.

  It remains to show that $f$ gives an isomorphism of F-manifolds.
  Note first that $T_{K(t)} = T_{K'(f(t))}$: theorem~\ref{thm:cone}
  implies that $K(t) \in z T_{K(t)}$, so $K'(f(t))$ is also in $z
  T_{K(t)}$, and so $T_{K(t)} = T_{K'(f(t))}$ by
  theorem~\ref{thm:cone} again.  Write $T = T_{K(t)} = T_{K'(f(t))}$.
  Using proposition~\ref{pro:TzTbasis}, we can write $w(t) \in T$
  uniquely in the form
  \begin{equation}
    \label{eq:defofg}
    w(t) = \nabla_{g(t)} K(t) + z h(t)
  \end{equation}
  for some vector field $g$ on $M_0$ and some element $h(t) \in T$.
  Thus for any vector field $v$ on $M_0$,
  \begin{align}
    \nabla_{f_\star v(t)} K'(f(t)) + z T &= \notag
    \nabla_{v(t)} \big( K(t) + z w(t) \big) + z T \\
    &= \notag
    \nabla_{v(t)} K(t) + z \nabla_{v(t)} \nabla_{g(t)} K(t) + z T \\
    &= \label{eq:characterizepushforward}
    \nabla_{v(t)} K(t) + \nabla_{v(t) \circ_t g(t)} K(t) + z T.
  \end{align}
  As the maps $DK|_t:T_t M_0 \to T/zT$ and $DK'|_{f(t)}:T_{f(t)} M'_0
  \to T/zT$ are isomorphisms, equation
  \eqref{eq:characterizepushforward} determines the pushforward
  $f_\star v$.  Differentiating again, along a vector field $w$ on
  $M_0$, gives
  \[
  z \nabla_{f_\star v(t)} \nabla_{f_\star w(t)} K'(f(t)) + z T =
  z \nabla_{v(t)} \nabla_{w(t)} K(t) +
  z \nabla_{w(t)} \nabla_{v(t) \circ_t g(t)} K(t) + z T,
  \]
  and hence
  \[
  \nabla_{(f_\star v(t)) \circ'_{f(t)} (f_\star w(t))} K'(f(t)) + zT =
  \nabla_{v(t) \circ_t w(t)} K(t) +
  \nabla_{v(t) \circ_t w(t) \circ_t g(t)} K(t) + z T.
  \]
  Comparing with \eqref{eq:characterizepushforward}, we find
  \[
  f_\star \Big( v(t) \circ_t w(t) \Big) =
  \Big( f_\star v(t) \Big) \circ'_{f(t)}
  \Big( f_\star w(t) \Big).
  \]
  The map $f$ is certainly invertible (this follows from uniqueness)
  and so $f$ gives an isomorphism of F-manifolds.
\end{proof}

\begin{rem}
  It was pointed out to us by Hiroshi Iritani that the arguments in
  this section show that the moduli space of tangent spaces to $\cLZ$
  carries a canonical F-manifold structure; see
  \cite{CCIT:crepant1}*{\S2.2} for a different point of view on this.
\end{rem}
\subsection{From Givental's Cone to a Frobenius Manifold}

Consider $\cLZ \subset \cHZ$ satisfying the conclusions of
theorem~\ref{thm:cone}, and $x \in \cLZ$.  As before, write $T_x = T_x
\cLZ$.  To construct a Frobenius manifold, we need to choose also an
\emph{opposite subspace} at $x$.

\begin{definition}
  Let $x \in \cLZ$.  A subspace $\cHopp \subset \cHZ$ is called
  \emph{opposite at $x$} or \emph{opposite to $T_x$} if $\cHopp$ is
  Lagrangian, $T_x \oplus \cHopp = \cHZ$, and $z^{-1} \cHopp \subset
  \cHopp$.
\end{definition}

For example, $\cHZ^-$ is opposite at $x$ for all $x \in \cLZ$.  Our
Frobenius manifold will be based on a formal neighbourhood of zero in
$z \cHopp/\cHopp$.

We note the following immediate consequence of oppositeness.

\begin{lemma} \label{lem:transversality}
  If $\cHopp$ is opposite to $T_x$ then the projections
  \begin{equation}
    \label{eq:defofpi}
    \xymatrix{ & z \cHopp \cap T_x \ar[ld] \ar[rd]^\pi & \\
      T_x/z T_x & & z \cHopp/\cHopp}
  \end{equation}
  are both isomorphisms. \qed
\end{lemma}

Consider the `slice' $\big(x + z \cHopp\big) \cap \cLZ$.  This is the
germ (at $x$) of a finite-dimensional submanifold of $\cLZ$, and
lemma~\ref{lem:transversality} implies that the map
\begin{equation}
  \label{eq:defofp}
  \begin{aligned}
    p: \big(x + z \cHopp\big) \cap \cLZ & \longrightarrow z
    \cHopp/\cHopp \\
    y & \longmapsto y - x + \cHopp
  \end{aligned}
\end{equation}
has bijective derivative at $x$.  Thus there is a map from the formal
neighbourhood $N_0$ of zero in $z\cHopp/\cHopp$,
\begin{equation}
  \label{eq:defofJ}
  J:N_0 \longrightarrow \big(x + z \cHopp\big) \cap \cLZ
\end{equation}
such that $p \circ J = \id$.  If we identify $N_0$ with a formal
neighbourhood of the origin in $z\cHopp \cap T_x$ via the isomorphism
$\pi$ in \eqref{eq:defofpi}, then
\[
J(t) = x + t + h(t)
\]
for some $h(t) \in \cHopp$, and so $J$ coincides with the map $K$
defined in \S\ref{sec:conetoQC}(b) by taking $V = \cHopp$.

As in \S\ref{sec:conetoQC}(b), the derivative $DJ|_t:T_t N_0 \to
T_{J(t)}/z T_{J(t)}$ is an isomorphism for all $t \in N_0$.  Pick a
basis $e_0,\ldots,e_N$ for $z \cHopp \cap T_x$ and denote the
corresponding linear co-ordinates on $N_0$, produced using
lemma~\ref{lem:transversality}, by $t_a$, $0 \leq a \leq N$.  Pulling
back the Frobenius algebra structure on $T_{J(t)}/z T_{J(t)}$ defined
in \S\ref{sec:conetoQC}(a) along the map $DJ$ gives a pairing
\[
g_{\alpha \beta}(t) =
\Omega\bigg(
{\partial J \over \partial t_\alpha}(t),
z^{-1} {\partial J \over \partial t_\beta}(t)
\bigg)
\]
and a symmetric $3$-tensor
\[
c_{\alpha \beta \gamma}(t) =
\Omega\bigg(
{\partial^2 J \over \partial t_\beta \partial t_\gamma}(t),
{\partial J \over \partial t_\alpha}(t)
\bigg)
\]
on $T_t N_0$.  We again denote the corresponding product on $T_t N_0$
by $\circ_t$ and the identity vector field, constructed in
proposition~\ref{pro:identity}, by $e$.  As before the product
$\circ_t$ can be determined by differentiating $J(t)$, but this time
the relationship between $\circ_t$ and $J(t)$ is more direct:

\begin{proposition} \label{pro:betterQDEs}
  $\nabla_{u \circ_t v} J(t) = - z \nabla_u \nabla_v J(t)$.
\end{proposition}

\begin{proof}
  Proposition~\ref{pro:Fprodisassociative}(a) shows that the quantity
  \begin{equation}
    \label{eq:iszero}
    \nabla_{u \circ_t v} J(t) + z \nabla_u \nabla_v J(t)
  \end{equation}
  lies in $z T_{J(t)}$.  On the other hand $J(t) = x + t + h(t)$,
  where $t \in z\cHopp \cap T_x$ and $h(t) \in \cHopp$, so
  \eqref{eq:iszero} lies in $z\cHopp$.  As $z \cHopp \cap z T_{J(t)} =
  \{0\}$ for all $t \in N_0$, the statement follows.
\end{proof}

\begin{proposition}
  The quadruple $(N_0,\circ,e,g)$ is a Frobenius manifold.  In other
  words:
  \begin{itemize}
  \item[(a)] each tangent space $(T_t N_0, \circ_t)$ is a unital
    supercommutative Frobenius algebra;
  \item[(b)] the metric $g_{\alpha \beta}(t)$ is flat and the
    co-ordinates $t_0,\ldots,t_N$ are flat co-ordinates;
  \item[(c)] the identity vector field $e$ is flat;
  \item[(d)] $c_{\alpha \beta \gamma}(t)$ is the third derivative of
    some function $\phi:N_0 \to \Lambda$.
  \end{itemize}
\end{proposition}

\begin{proof}
  Part (a) was proved in \S\ref{sec:conetoQC}(b).  Part (d) is
  immediate from the construction of the tensor $c$.  For (b) we have
  \begin{align} \label{eq:partialJ}
    {\partial J \over \partial t_\alpha}(t) = e_\alpha + h_\alpha(t),
    &&
    \text{where $e_\alpha \in z \cHopp$ and $h_\alpha(t) \in \cHopp$,}
  \end{align}
  and so
  \[
  g_{\alpha \beta}(t) =
  \Omega\Big(
  e_\alpha + h_\alpha(t),
  z^{-1} e_\beta + z^{-1} h_\beta(t)
  \Big).
  \]
  As $\cHopp$ is Lagrangian and $z^{-1} \cHopp \subset \cHopp$,
  $g_{\alpha \beta}(t) = \Omega(e_\alpha,e_\beta)$ is independent of
  $t$.  This shows that $g$ is flat, and that $\{t_a\}$ are flat
  co-ordinates.

  For (c) we need to show that $e(t)$ is constant in flat
  co-ordinates.  In view of \eqref{eq:partialJ}, we need to show that
  $\nabla_{e(t)}J(t) + \cHopp$ is constant with respect to $t$.
  Proposition~\ref{pro:betterQDEs} shows that $z \nabla_{e(t)}
  \nabla_{v(t)} J(t) = \nabla_{v(t)} J(t)$ for any vector field $v$ on
  $N_0$, and hence that $\nabla_{e(t)} J(t) = z^{-1} J(t) + C$ for
  some $C$ independent of $t$.  Thus
  \begin{align*}
    \nabla_{e(t)} J(t) + \cHopp &= z^{-1}\big( x + t + h(t) \big) + C
    + \cHopp \\
    &= z^{-1} x + C + \cHopp
  \end{align*}
  is independent of $t$.  This completes the proof.
\end{proof}

\subsection{Example: the Quantum Cohomology of $\cX$}

We now show that if we take $x$ to be the point $\cLX \cap
\big({-z} + \cHX^-\big)$ and set $\cHopp = \cHX^-$, then the Frobenius
manifold constructed in the previous section is the quantum cohomology
Frobenius manifold of $\cX$.  Set $\tau = \tau_\alpha \phi_\alpha$,
and consider the element $J_\cX(\tau,-z)$ of $\cLX$ such that its
projection to $\cHX^+$ along $\cHX^-$ is equal to $-z + \tau$.  We
call $J_\cX(\tau,-z)$ the \emph{$J$-function of $\cX$}.  It is
obtained by substituting $\tau_{0,a} = \tau_a$, $0 \leq a \leq N$;
$\tau_{k,a} = 0$, $0 \leq a \leq N$, $0 < k < \infty$; and
\[
p_{l,b} = \left.\parfrac{\cF^0_\cX}{\tau_{l,b}}\right|_{\btau(z) = \tau} =
\sum_{d \in \NE(\cX)}
\sum_{n \geq 0}
\correlator{\tau,\ldots,\tau,\phi_b \psi^l}^\cX_{0,n+1,d} {U^d \over n!}
\]
into \eqref{eq:Darboux}, via \eqref{eq:dilatonshiftX}.  Thus
\[
J_\cX(\tau,-z) = -z + \tau +
\sum_{d \in \NE(\cX)}
\sum_{n \geq 0}
\sum_{l \geq 0}
\correlator{\tau,\ldots,\tau,\phi_\epsilon \psi^l}^\cX_{0,n+1,d} {U^d
  \phi^\epsilon \over n! (-z)^{l+1}};
\]
we abbreviate this to
\[
J_\cX(\tau,-z) = -z + \tau +
\sum_{d \in \NE(\cX)}
\sum_{n \geq 0}
\correlator{\tau,\ldots,\tau,{\phi_\epsilon \over -z-\psi}}^\cX_{0,n+1,d} {U^d
  \phi^\epsilon \over n!}.
\]
$J_\cX(\tau,-z)$ is an element of $\cLX$ --- a formal power series in
variables $\tau_0, \ldots, \tau_N$ taking values in $\cHX$ --- which
depends analytically on $\tau_1,\ldots,\tau_s$ in the domain $\CC^s$.
We can see this analyticity explicitly:

\begin{proposition} \label{pro:divisorX}
  \begin{multline*}
    J_\cX(\tau,-z) = \re^{- \tau_{\rm two}/z} \times \\
    \Bigg(
    {-z} + \tau_{\rm rest} +
    \sum_{d \in \NE(\cX) }
    \sum_{n \geq 0}
    \correlator{\tau_{\rm rest},\ldots,\tau_{\rm rest},{\phi_\epsilon \over -z-\psi}}^\cX_{0,n+1,d} {U^d
      \re^{d_1 \tau_1} \cdots \re^{d_s \tau_s} \phi^\epsilon \over n!} \Bigg)
  \end{multline*}
  where $\tau_{\rm two}$ and $\tau_{\rm rest}$ are defined in \eqref{eq:tworest}.
\end{proposition}

\begin{proof}
  This follows easily from the Divisor Equation, as in \cite{CCLT}*{lemma~2.5}.
\end{proof}

Our Frobenius manifold is based on a formal neighbourhood $N_0(\cX)$
of the origin in $z\cHX^-/\cHX^- \cong \HorbXNovZ$.  Choose a point $x
\in \cLX \cap \big({-z} + z \cHX^- \big)$ and write $x =
-z+\sigma+h_-$ with $\sigma \in \HorbXNovZ$ and $h_- \in \cHX^-$.  Then
the map $p$ defined in \eqref{eq:defofp} satisfies
\[
p \circ J_\cX(\sigma + \tau,-z) = \tau,
\]
and so the map $J$ defined in \eqref{eq:defofJ} is
\[
J(\tau) = J_{\cX}(\sigma + \tau,-z).
\]
The basis $\phi_0,\ldots,\phi_N$ for $\HorbXNovZ$ gives co-ordinates
$\tau_a$, $0 \leq a \leq N$, on $N_0(\cX)$ and these are flat
co-ordinates for the Frobenius manifold:
\begin{align*}
g_{\alpha \beta}(\tau) & =
\Omega\bigg(
{\partial J_{\cX} \over \partial \tau_\alpha}(\tau + \sigma,-z),
z^{-1} {\partial J_{\cX} \over \partial \tau_\beta}(\tau + \sigma,-z)
\bigg) \\
&= \Omega\big(\phi_\alpha + h_\alpha, z^{-1} \phi_\beta + z^{-1}
h_\beta\big) &&
\text{where $h_\alpha, h_\beta \in \cHX^-$} \\
&= \big(\phi_\alpha,\phi_\beta\big)_{\cX}.
\end{align*}
To calculate the structure constants of the product $\circ_\tau$, we
will need
\begin{align*}
  & {\partial J_\cX \over \partial \tau_\alpha}(\sigma + \tau) =
  \phi_\alpha + h_\alpha \\
  & {\partial^2 J_\cX \over \partial \tau_\beta \partial
    \tau_\gamma}(\sigma + \tau)
  =- z^{-1} \sum_{d \in \NE(\cX)} \sum_{n \geq 0}
  \correlator{\phi_\beta,\phi_\gamma,\sigma + \tau,\ldots,\sigma + \tau,\phi_\epsilon}^\cX_{0,n+3,d}
  {U^d \phi^\epsilon \over n!}  \\
  & \qquad \qquad \qquad \qquad + z^{-1} h_{\beta \gamma}
\end{align*}
for some $h_\alpha, h_{\beta \gamma} \in \cHX^-$; this gives
\begin{align*}
c_{\alpha \beta \gamma}(\tau) & =
\Omega\bigg(
{\partial^2 J_\cX \over \partial \tau_\beta \partial
  \tau_\gamma}(\sigma + \tau),
{\partial J_\cX \over \partial \tau_\alpha}(\sigma + \tau)
\bigg) \\
&= \sum_{d \in \NE(\cX)} \sum_{n \geq 0}
\correlator{\phi_\beta,\phi_\gamma,\sigma + \tau,\ldots,\sigma + \tau,\phi_\alpha}^\cX_{0,n+3,d}
\\
&= {\partial^3 F_\cX \over \partial \tau_\alpha \partial \tau_\beta
  \partial \tau_\gamma} (\sigma + \tau).
\end{align*}
Thus the product $\circ_{\tau}$ on the Frobenius manifold is a shifted
version of the big quantum product for $\cX$:
\begin{equation}
  \label{eq:shifted}
  v \circ_{\tau} w = v \QC{\sigma + \tau} w.
\end{equation}
We have proved:
\begin{proposition}
  The Frobenius manifold produced from $\cLX \subset \cHX$ by choosing
  $x = \cLX \cap \big({-z} + \sigma + \cHX^-\big)$, where $\sigma \in
  \HorbXNovZ$, and $\cHopp = \cHX^-$ is the Frobenius manifold
  corresponding to the quantum cohomology of $\cX$ with the product
  `shifted' by $\sigma$.  It has flat metric given by the orbifold
  Poincar\'e pairing $\( \cdot, \cdot \)_\cX$ and product given by the
  shifted big quantum product \eqref{eq:shifted}.  In particular,
  choosing $\sigma = 0$ gives the usual quantum cohomology Frobenius
  manifold for $\cX$. \qed
\end{proposition}

For later use, we note a stronger version of
proposition~\ref{pro:TzTbasis}:

\begin{proposition}
  For all $\tau \in N_0(\cX)$, the elements
  \begin{align*}
    {\partial J_\cX \over
      \partial \tau_a}(\tau,-z) &&a = 0,1,\ldots,N
  \end{align*}
  form a $\Lambda[z]$-basis for $T_{J_\cX(\tau,-z)}$.
\end{proposition}

\begin{proof}
  Every element of $T_{J_\cX(\tau,-z)}$ can be uniquely written in the
  form $h_+ + h_-$ for $h_+ \in \cHX^+$, $h_- \in \cHX^-$.  The
  element $h_+$ is a polynomial in $z$.  Since ${\partial J_\cX \over
    \tau_a}(\tau,-z) = \phi_a + h_-'$ for some $h_-' \in \cHX^-$,
  since $\{\phi_a\}$ is a $\Lambda$-basis for $\HorbXNovZ$, and since
  $T_{J_\cX(\tau,-z)}$ is closed under multiplication by $z$, the
  result follows by induction on the degree of $h_+$.
\end{proof}

We will also need to know the behaviour of $J_\cX(\tau,-z)$ as $\tau$
approaches the large radius limit point of $\cX$.

\begin{proposition} \label{pro:LRL}
  Write $\tau = \tau_{\rm two} + \tau_{\rm rest}$, as in
  \eqref{eq:tworest}.  As $\tau$ approaches the large radius limit
  point for $\cX$,
  \[
  \begin{aligned}
    \Real \tau_i &\to -\infty, && 1 \leq i \leq s,\\
    \tau_i &\to 0, && \text{$i=0$ and $s<i\leq N$,}
  \end{aligned}
  \]
  $J_\cX(\tau,-z) \to {-z}\re^{-\tau_{\rm two}/z}$ and the tangent space
  $T_{J_{\cX}(\tau,-z)} \to \re^{-\tau_{\rm two}/z} \cHX^+$.
\end{proposition}

\begin{proof}
  Look at proposition~\ref{pro:divisorX}.  As $\tau$ approaches the
  large radius limit point, all terms in $J_\cX(\tau,-z)$ with $d \ne
  0$ and all terms involving $\tau_{\rm rest}$ vanish.  Thus
  \begin{align*}
    J_\cX(\tau,-z) \to {-z} \re^{- \tau_{\rm two}/z} && \text{and} &&
    {\partial J_\cX \over \partial \tau_a} (\tau,-z) \to \phi_a \re^{- \tau_{\rm two}/z} .
  \end{align*}
  As $T_{J_\cX(\tau,-z)}$ is the $\Lambda[z]$-span of $\Big\{
  {\partial J_\cX \over \partial \tau_a} (\tau,-z) : 0 \leq a \leq
  N\Big\}$, it follows that
  \[
  T_{J_\cX(\tau,-z)} \to \re^{-\tau_{\rm two}/z} \cHX^+.
  \]
\end{proof}

\subsection{Example: the Modified Quantum Cohomology of $Y$}

We now show that, as one might expect, the Frobenius manifold
constructed from $\cLY \subset \cHY$ by choosing $x \in \cLY \cap
\big( -z + z \cHY^-\big)$ and $\cHopp = \cHY^-$ is the Frobenius
manifold based on the modified big quantum product $\circledast$ for
$Y$.  The argument is very similar to that in the previous section,
but there are some additional complications caused by our having made
the substitution
\begin{equation}
  \label{eq:substagain}
  Q_i =
  \begin{cases}
    U_i & 1 \leq i \leq s \\
    1 & s< i \leq r.
  \end{cases}
\end{equation}

Set $t = t_\alpha \varphi_\alpha$ and let $t_{\rm two}$ and $t_{\rm
  rest}$ be as in \eqref{eq:tworest}.  Consider the element
$J_Y^\circledast(t,-z)$ of $\cLY$ such that its projection to $\cHY^+$
along $\cHY^-$ is equal to $-z + t$.  This is the \emph{modified
  $J$-function} of $Y$.  It is obtained by setting $t_{0,a} = t_a$, $0
\leq a \leq N$; $t_{k,a} = 0$, $0 \leq a \leq N$, $0 < k < \infty$;
and
\[
p_{l,b} = \left.\parfrac{\cF^0_Y}{t_{l,b}}\right|_{\bt(z) = t} =
\sum_{d \in \NE(Y)}
\sum_{n \geq 0}
\correlator{t,\ldots,t,\varphi_b \psi^l}^Y_{0,n+1,d} {Q^d \over n!}
\]
in \eqref{eq:Darboux}, and then
making the substitution \eqref{eq:substagain}.  Before making the
substitution \eqref{eq:substagain} we have
\[
-z + t +
\sum_{d \in \NE(Y)}
\sum_{n \geq 0}
\correlator{t,\ldots,t,{\varphi_\epsilon \over -z-\psi}}^Y_{0,n+1,d} {Q^d
  \varphi^\epsilon \over n!}
\]
and using the Divisor Equation, as in proposition~\ref{pro:divisorX},
we can write this as
\begin{multline*}
\re^{- t_{\rm two}/z}
\Bigg(
{-z} + t_{\rm rest} + \\
\sum_{d \in \NE(Y)}
\sum_{n \geq 0}
\correlator{t_{\rm rest},\ldots,t_{\rm rest},{\varphi_\epsilon \over
    -z-\psi}}^Y_{0,n+1,d}
{Q^d \re^{d_1 t_1} \cdots \re^{d_r t_r} \varphi^\epsilon \over n!}
\Bigg).
\end{multline*}
Thus
\begin{multline*}
  J_Y^\circledast(t,-z) = \re^{- t_{\rm two}/z}
  \Bigg(
  {-z} + t_{\rm rest} + \\
  \sum_{d \in \NE(Y)}
  \sum_{n \geq 0}
  \correlator{t_{\rm rest},\ldots,t_{\rm rest},{\varphi_\epsilon \over
      -z-\psi}}^Y_{0,n+1,d}
  {U_1^{d_1} \cdots U_s^{d_s} \re^{d_1 t_1} \cdots \re^{d_r t_r} \varphi^\epsilon \over n!}
  \Bigg)
\end{multline*}
where $d = d_1 \beta_1 + \cdots + d_r \beta_r$.  The modified
$J$-function $J_Y^\circledast(t,-z)$ is an element of $\cLY$ which
depends formally on the variables $t_0$, $t_{r+1}, t_{r+2},\ldots,
t_N$ and analytically on $t_1,\ldots,t_r$ in the domain
\eqref{eq:strangeregion}.  It is the unique element of $\cLY$ of the
form
\[
\text{$-z + t + h_-(t)$ \quad with $h_-(t) \in \cHY^-$.}
\]

The Frobenius manifold we seek is based on a formal neighbourhood
$N_0(Y)$ of the origin in $z\cHY^-/\cHY^- \cong \HYNovZ$.  Choose a
point $x \in \cLY \cap \big({-z} + z \cHY^- \big)$ and write $x =
-z+s+h'_-$ with $s \in \HYNovZ$ and $h'_- \in \cHY^-$.  Then the map
$p$ defined in \eqref{eq:defofp} satisfies
\[
p \circ J_Y^\circledast(s+t,-z) = t,
\]
and so the map $J$ defined in \eqref{eq:defofJ} is
\[
J(t) = J_Y^\circledast(s+t,-z).
\]
Now, using the co-ordinates $t_0,\ldots, t_N$ given by the basis
$\varphi_0,\ldots,\varphi_N$ for $\HYNovZ$ and arguing exactly as in
\S\ref{sec:conetoQC}(d), we find that the flat metric on $N_0(Y)$ is
given by the Poincar\'e pairing:
\[
g_{\alpha \beta}(t) = \big(\varphi_\alpha,\varphi_\beta\big)_Y
\]
and that the structure constants of the product $\circ_\tau$ are
\[
c_{\alpha \beta \gamma}(t) =
{\partial^3 F_Y^\circledast \over \partial t_\alpha \partial t_\beta
  \partial t_\gamma} (s+t).
\]
Thus the product $\circ_{\tau}$ on the Frobenius manifold $N_0(Y)$ is
a shifted version of the modified big quantum product for $Y$:
\begin{equation}
  \label{eq:shiftedY}
  v \circ_{t} w = v \newQC{s+t} w.
\end{equation}
We have proved:
\begin{proposition} \label{pro:FrobY} The Frobenius manifold produced
  from $\cLY \subset \cHY$ by choosing $x = \cLY \cap \big( {-z} + s +
  \cHY^-\big)$, for some $s \in \HYNovZ$, and $\cHopp = \cHY^-$ is the
  Frobenius manifold corresponding to the modified quantum cohomology
  of $Y$ with the product `shifted' by $s$.  It has flat metric given
  by the Poincar\'e pairing $\( \cdot, \cdot \)_Y$ and product given
  by \eqref{eq:shiftedY}. \qed
\end{proposition}

\begin{rem} \label{rem:twoproducts} We now explain why condition (c)
  in conjecture~\ref{conj} ensures that there is a neighbourhood of
  the large-radius limit point for $\cX$ in which both the big quantum
  product $\star$ for $\cX$ and the analytic continuation of the
  modified big quantum product $\circledast$ for $Y$ are well-defined.
  Let us write $V_1 \pitchfork V_2$ if and only if $V_1 \oplus V_2 =
  \cHX$, so that condition (c) is the assertion $\cHX^+ \pitchfork
  \U^{-1}(\cHY^-)$.  In \S\ref{sec:conetoQC}(d) we saw that by
  choosing $x \in \cLX$ of the form $x = - z + \sigma + h_-$, where
  $\sigma \in \HorbXNovZ$ and $h_- \in \cHX^-$, and taking opposite
  subspace $\cHopp = \cHX^-$ we obtain a Frobenius manifold with
  product a shifted version of the big quantum product for $\cX$: $v
  \circ_\tau w = v \QC{\sigma + \tau} w$.  Suppose now that
  conjecture~\ref{conj} holds.  In proposition~\ref{pro:FrobY} we saw
  that by choosing $y \in \cLY$ of the form $-z + s + h_-'$, where $s
  \in \HYNovZ$ and $h_-' \in \cHY^-$, and taking opposite subspace
  $\cHopp = \cHY^-$ we obtain a Frobenius manifold with product
  $v \circ_t w = v \newQC{s+t} w$.  The analytic continuation of $\cLY$ chosen
  as part of conjecture~\ref{conj} defines, via
  proposition~\ref{pro:FrobY}, an analytic continuation of the product
  $\newQC{s+t}$.  (Here we analytically continue $\newQC{s+t}$ in $s$;
  the variable $s$ determines and is determined by the basepoint $y =
  -z + s + h_-' \in \cLY$.)\phantom{.}  We can obtain this
  analytically continued product either by choosing $y$ in the
  analytic continuation of $\cLY$ and taking opposite subspace $\cHopp
  = \cHY^-$ or --- and this is equivalent via $y = \U(x)$ --- by
  choosing $x \in \cLX$ and taking opposite subspace $\cHopp =
  \U^{-1}(\cHY^-)$.  For this to give a Frobenius manifold, we need
  $\U(\cHY^-)$ to be opposite to $T_x = T_x \cLX$; in other words we
  need $T_x \pitchfork \U^{-1}(\cHY^-)$.  Let $x = \cLX \cap \big({-z}
  + \sigma + \cHX^-\big)$.  We know from proposition~\ref{pro:LRL}
  that as $\sigma$ approaches the large-radius limit point for $\cX$,
  $T_x \to \re^{-\sigma_{\rm two}/z} \cHX^+$.  But
  \begin{align*}
    \Big(\re^{-\sigma_{\rm two}/z} \cHX^+ \Big)\pitchfork \U^{-1}\big(\cHY^-\big)
    & \iff
    \cHX^+\pitchfork \re^{\sigma_{\rm two}/z} \U^{-1}\big(\cHY^-\big)
    \\
    & \iff
    \cHX^+\pitchfork  \U^{-1} \Big( \re^{\pi^\star \sigma_{\rm
        two}/z} \cHY^-\Big) \\
    & \iff
    \cHX^+\pitchfork  \U^{-1} \big(\cHY^-\big),
  \end{align*}
  and this holds by conjecture~\ref{conj}(c).  Thus for $\sigma$ in a
  neighbourhood of the large-radius limit point for $\cX$, $T_x
  \pitchfork \U^{-1}(\cHY^-)$ and so both the Frobenius manifold
  defined by the big quantum product for $\cX$ (basepoint = $x \in
  \cLX$, $\cHopp = \cHX^-$) and the Frobenius manifold defined by the
  analytic continuation of the modified big quantum product for $Y$
  (basepoint = $x$, $\cHopp = \U^{-1}(\cHY^-)$) are well-defined.
\end{rem}
\section{A Version of the Cohomological Crepant Resolution Conjecture}
\label{sec:CCRC}

The Cohomological Crepant Resolution Conjecture
\cite{Ruan:firstconjecture} describes a relationship between the
Chen--Ruan cohomology ring of $\cX$ and the small quantum cohomology
ring of the crepant resolution $Y$.  Conjecture~\ref{conj} implies
such a relationship, as we now explain.  The family of Frobenius
algebras constructed in \S\ref{sec:conetoQC}(a) depends only on the
submanifold-germ $\cLZ$ and the symplectic space $\cHZ$. The
transformation $\U$ from conjecture~\ref{conj}, which is a
$\CC(\!(z)\!)$-linear symplectic isomorphism and satisfies $\U(\cLZ) =
\cLY$, therefore induces an isomorphism between the families of
Frobenius algebras
\begin{align*}
  T \cLX / z T \cLX \to \cLX
  && \text{and} &&
  T \cLY / z T \cLY \to \cLY
\end{align*}
By choosing $x \in \cLX$ appropriately --- by taking $x = \cLX \cap
\big( {-z} + \sigma + \cHX^- \big)$ and letting $\sigma$ approach the
large-radius limit point for $\cX$ --- we can obtain the Chen--Ruan
cohomology of $\cX$ as the Frobenius algebra $T_x/z T_x$.  Let $y \in
\cLY$ be such that $y = \U(x)$, and let $T_y$ denote the tangent space
$T_y \cLY$.  Then $\U$ induces an isomorphism of Frobenius algebras
$T_x/z T_x \cong T_y / z T_y$, and this expresses the Chen--Ruan
cohomology ring of $\cX$ in terms of the quantum cohomology of $Y$.

Let $\sigma \in H^2(\cX;\CC)$ and let $x = \cLX \cap \big({-z} +
\sigma + \cHX^-\big)$.  Then $T_x/z T_x$ is isomorphic as a Frobenius
algebra to the quantum cohomology of $\cX$,
$\big(\HorbXNovZ,\QC{\sigma}\,\big)$.  As $\sigma$ approaches the
large-radius limit point for $\cX$, therefore, $T_x/zT_x$ approaches
the Chen--Ruan cohomology ring $\big(\HorbXNovZ,\CR\big)$ --- see the
discussion below equation~\ref{eq:bigQCXdivisor}.  Let $y = \U(x)$.

\begin{proposition}
  As $\sigma$ approaches the large-radius limit point for $\cX$
  \[
  y \to J_Y(\pi^\star \sigma + c, -z),
  \]
  where $\U(\fun_\cX) = \fun_Y - c z^{-1} + O(z^{-2})$.
\end{proposition}

\begin{proof}
  We have $x = J_\cX(\sigma,-z)$ so, by proposition~\ref{pro:LRL}, $x
  \to {-z} \re^{- \sigma/z}$ as $\sigma$ approaches the large-radius
  limit point for $\cX$.  Thus
  \begin{align*}
    y & \to \U\big({-z}\re^{-\sigma/z}\big) \\
      &= {-z} \re^{\pi^\star \sigma/z} \U(\fun_\cX) && \text{by
        conjecture~\ref{conj}(b)} \\
      &= -z + \pi^\star \sigma + c + h_- && \text{for some $h_- \in \cHX^-$.}
  \end{align*}
  There is a unique point on $\cLY$ of the form $-z + \pi^\star \sigma
  + c + h_-$, $h_- \in \cHX^-$, and that is $J_Y(\pi^\star \sigma +
  c,-z)$.  Thus as $\sigma$ approaches the large-radius limit point
  for $\cX$, $y \to J_Y(\pi^\star \sigma + c,-z)$.
\end{proof}

It follows that as $\sigma$ approaches the large-radius limit point
for $\cX$,
\begin{equation}
  \label{eq:LRLX}
  \begin{aligned}
    \Real \sigma_i &\to -\infty, && 1 \leq i \leq s,
  \end{aligned}
\end{equation}
the Frobenius algebra $T_y/z T_y$ approaches the quantum cohomology
algebra
\begin{equation}
  \label{eq:CCRClimit}
  \lim_{\substack{\Real \sigma_i \to -\infty,\\ 1 \leq i \leq s}} \Big(
  \HYNovZ, \newQC{ \pi^\star \sigma + c}
  \Big).
\end{equation}
By assumption $\U$ is grading preserving and so $c \in H^2(Y;\CC)$;
let us write $c = c_1 \varphi_1 + \ldots + c_r \varphi_r$.  Note that
there is analytic continuation hidden in \eqref{eq:CCRClimit}: if $t =
t_1 \varphi_1 + \ldots + t_r \varphi_r \in H^2(Y;\CC)$ then the
product $\newQC{t}$ is defined as a power series
\eqref{eq:newQCYdivisor} which converges only when $|\re^{t_i}|<R_i$,
$s < i \leq r$.  In general $t = \pi^\star \sigma + c$ will be outside
this domain of convergence.  But the analytic continuation of $\cLY$
defines, via proposition~\ref{pro:FrobY}, an analytic continuation of
the product $\newQC{t}$ and it is this analytically-continued product
which we use in \eqref{eq:CCRClimit}.  We compute the limit
\eqref{eq:CCRClimit} as follows.  From \eqref{eq:newQCYdivisor} we
have
\[
\varphi_\alpha \newQC{t} \varphi_\beta =
\sum_{\substack{d \in \NE(Y): \\
      d = d_1 \beta_1 + \cdots + d_r \beta_r}}
\correlator{\varphi_\alpha,\varphi_\beta,\varphi^\epsilon}^Y_{0,3,d}
U_1^{d_1} \cdots U_s^{d_s}
\re^{d_1 t_1} \cdots \re^{d_r t_r} \varphi_\epsilon
\]
whenever $|\re^{t_i}|<R_i$ for $s<i\leq r$; taking the limit $\Real
t_i \to - \infty$, $1 \leq i \leq s$, gives
\begin{equation}
  \label{eq:CCRCintermediate}
  \varphi_\alpha \newQC{t} \varphi_\beta \to
  \sum_{\substack{d \in \ker \pi_\star: \\
      d = d_{s+1} \beta_{s+1} + \cdots + d_r \beta_r}}
  \correlator{\varphi_\alpha,\varphi_\beta,\varphi^\epsilon}^Y_{0,3,d}
  \re^{d_{s+1} t_{s+1}} \cdots \re^{d_r t_r} \varphi_\epsilon.
\end{equation}
We can obtain the algebra \eqref{eq:CCRClimit} which we seek from
\eqref{eq:CCRCintermediate} by analytic continuation in
$t_{s+1},\ldots,t_r$ followed by the substitution $t_i = c_i$, $s < i
\leq r$.  This proves:
\begin{theorem} \label{thm:CCRC}
  If conjecture~\ref{conj} holds then the Chen--Ruan product $\CR$ on
  $\HorbX$ can be obtained from the small quantum product
  \eqref{eq:smallQC} for $Y$ by analytic continuation in the quantum
  parameters $Q_{s+1},\ldots,Q_r$ (if necessary) followed by the
  substitution
  \begin{equation}
    \label{eq:CCRCsubst}
    Q_i =
    \begin{cases}
      0 & 1 \leq i \leq s \\
      \re^{c_i} & s < i \leq r.
    \end{cases}
  \end{equation}

\end{theorem}

The small quantum cohomology with quantum parameters $Q_i$ specialized
like this is known as \emph{quantum corrected cohomology}
\cite{Ruan:firstconjecture}. In Ruan's original Cohomological Crepant
Resolution Conjecture, the exceptional $Q_i$ were specialized to
$-1$. Calculations by Perroni \cite{Perroni} and
Bryan--Graber--Pandharipande \cite{Bryan--Graber--Pandharipande} have
shown that we must relax this, allowing the exceptional $Q_i$ to be
specialized to other roots of unity.  Here, we allow arbitrary choice.
It should be noted that the specialization $Q_i = \re^{c_i} =
\re^{\langle c, \beta_i \rangle }$ is independent of our choice of
bases (see \S\ref{sec:gerbes} for more on this).

\section{A Version of Ruan's Conjecture}
\label{sec:Ruan}

Ruan's original Crepant Resolution Conjecture (implicit in
\cite{Ruan:firstconjecture}), as modified in light of the calculations
of Perroni and Bryan--Graber--Pandharipande, was that the small
quantum cohomology algebra of the crepant resolution $Y$ becomes
isomorphic to the small quantum cohomology algebra of $\cX$ after
analytic continuation in the quantum parameters $Q_{s+1},\ldots,Q_r$
followed by a change-of-variables
\begin{equation}
  \label{eq:wrongcofv}
  Q_i =
  \begin{cases}
    \omega_i U_i & 1 \leq i \leq s \\
    \omega_i & s < i \leq r
  \end{cases}
\end{equation}
where the $\omega_i$ are roots of unity.  Conjecture~\ref{conj}
implies something very like this, at least when $\cX$ is
semi-positive, as we now explain.

\begin{definition}
  A K\"ahler orbifold $\cX$ is called \emph{semi-positive} if and only
  if there does not exist $d \in \NE(\cX)$ such that
  \[
  3 - \dim_\CC \cX \leq c_1(T\cX)\cdot d < 0.
  \]
\end{definition}

\noindent All K\"ahler orbifolds of complex dimension $3$ or less are
semi-positive, as are all Fano and Calabi--Yau orbifolds.
Semi-positive Gorenstein orbifolds $\cX$ have the property that if
$c_1(T\cX) \cdot d < 0$ then all genus-zero Gromov---Witten invariants
in degree $d$ vanish:

\begin{proposition}
  Suppose that $\cX$ is a semi-positive Gorenstein K\"ahler orbifold and
  that $\correlator{\delta_1 \psi^{a_1},\ldots,\delta_n
    \psi^{a_n}}^\cX_{0,n,d} \ne 0$.  Then $c_1(T\cX)\cdot d \geq 0$.
\end{proposition}

\begin{proof}
  Suppose not, so that $c_1(T\cX)\cdot d < 0$.  Without loss of
  generality we may assume that the marked points $1,2,\ldots,n'$
  carry classes $\delta_i$ from the twisted sectors and that the
  remaining marked points carry untwisted classes.  Let
  $\pi:\cX_{0,n,d} \to \cX_{0,n',d}$ be the map induced by forgetting
  all the untwisted marked points.  Then $\correlator{\delta_1 \psi^{a_1},\ldots,\delta_n
    \psi^{a_n}}^\cX_{0,n,d}$ is the degree-zero part of
  \begin{equation}
    \label{eq:doesntvanish}
    \big[ \cX_{0,n',d} \big]^{\rm vir} \cap \bigg(\prod_{k=1}^{n'}
    \ev_k^\star \delta_k \bigg)
    \cup \pi_\star \bigg(\prod_{k=n'+1}^n
    \ev_k^\star \delta_k \cup \prod_{k=1}^n \psi_k^{a_k} \bigg).
  \end{equation}
  As $\cX$ is Gorenstein, we know that $\deg \delta_k \geq 1$, $1 \leq
  k \leq n'$, where $\deg$ denotes the age-shifted degree on $\HorbX$.
  The non-vanishing of \eqref{eq:doesntvanish} therefore implies that
  the virtual (complex) dimension of $\cX_{0,n',d}$ is at least $n'$,
  and so
  \[
  n' + \dim_\CC \cX - 3 + c_1(T\cX) \cdot d \geq n'.
  \]
  It follows that
  \[
  3 - \dim_\CC \cX \leq c_1(T\cX) \cdot d < 0,
  \]
  which contradicts semi-positivity.  The proposition is proved.
\end{proof}

The small quantum cohomology of $\cX$ is the Frobenius algebra $\big(
\HorbXNovZ, \QC{\tau} \, \big)$ at $\tau = 0$.  This is the Frobenius
algebra $T_x/z T_x$ where $x = \cLX \cap \big({-z} + \cHX^-\big)$ and
$T_x = T_x \cLX$.  Let $y = \U(x)$ and $T_y = T_y \cLY$.  The map $\U$
induces an isomorphism between the Frobenius algebras $T_x/z T_x$ and
$T_y/z T_y$, and this isomorphism expresses the small quantum
cohomology of $\cX$ in terms of the quantum cohomology of $Y$.  To see
that it relates the small quantum cohomology of $\cX$ to the
\emph{small} quantum cohomology of $Y$, we need to calculate $y$.

\begin{proposition} \label{pro:RuanCRC} Suppose that $\cX$ is
  semi-positive and that conjecture~\ref{conj} holds.  Let $x = \cLX
  \cap \big({-z} + \cHX^-\big)$, and define $c \in H^2(Y;\CC)$ by
  $\U(\fun_\cX) = \fun_Y - c z^{-1} + O(z^{-2})$.  Then there is a
  unique element $f \in H^2(Y;\CC) \otimes \NovZ$, 
  \begin{align*}
    f = f_1 \varphi_1 + \cdots + f_r \varphi_r
    && \text{for some $f_1,\ldots,f_r \in \NovZ$,}
  \end{align*}
  such that $\U(x) = J_Y(c + f,-z)$.  Furthermore, the class $f$ is
  exceptional: $\pi_! f = 0$.
\end{proposition}

\begin{proof}
  Uniqueness is obvious.  For existence, we need to find $f \in
  H^2(Y;\CC) \otimes \NovZ$ such that
  \begin{equation}
    \label{eq:Ruanseek}
    \U(x) = -z + c + f +    h_-
  \end{equation}
  for some $h_- \in \cHY^-$.  We have $x = J_\cX(0,-z)$, so
  \begin{equation}
    x     = {-z} + \sum_{\substack{d \in \NE(\cX):\\d \ne 0}} \  \sum_{k \geq 0}
    (-1)^{k+1} \correlator{\phi^\epsilon \psi^k}^\cX_{0,1,d} U^d
    \, \phi_\epsilon z^{-k-1}.    \label{eq:Ruanx}
  \end{equation}
  If we set $\deg U^d = c_1(T\cX) \cdot d$, $\deg z = 2$, and give the
  Chen--Ruan class $\phi_\epsilon$ its age-shifted degree then $x \in
  \cHX$ is homogeneous of degree two.  As $\cX$ is semi-positive, any
  monomial $U^d$ which occurs in \eqref{eq:Ruanx} has non-negative
  degree, and so each term $\phi_\epsilon z^{-k-1}$ in
  \eqref{eq:Ruanx} has degree at most two.  If $\phi_\epsilon
  z^{-k-1}$ is of negative degree then $\U\big(\phi_\epsilon
  z^{-k-1}\big)$ is also of negative degree and so
  $\U\big(\phi_\epsilon z^{-k-1}\big) \in \cHY^-$.  If $\phi_\epsilon
  z^{-k-1}$ is of degree zero or one then, by parts (a) and (b) of
  lemma~\ref{lem:lowdegree}, $\U\big(\phi_\epsilon z^{-k-1}\big) \in
  \cHY^-$ as well.  If $\phi_\epsilon z^{-k-1}$ is of degree two then
  \[
  \U\big(\phi_\epsilon z^{-k-1}\big) = b_\epsilon + h_\epsilon
  \]
  for some exceptional class $b_\epsilon \in H^2(Y;\CC)$ and some
  $h_\epsilon \in \cHY^-$, by lemma~\ref{lem:lowdegree}(c).  Also, if
  $\phi_\epsilon z^{-k-1}$ is of degree two then $\deg \phi_\epsilon
  \geq 4$ and $k = {1 \over 2} w_\epsilon - 2$ where $w_\epsilon =
  \deg \phi_\epsilon$.  Thus
  \[
  \U(x) = {-z} + c + \sum_{\substack{d \in \NE(\cX): d \ne 0, \\
      c_1(T\cX)\cdot d = 0}} \
  \sum_{e = r+1}^N
  (-1)^{{1 \over 2}w_e+1} \correlator{\phi^e \psi^{{1 \over 2}w_e-2}}^\cX_{0,1,d} U^d \,
  b_e + h_-
  \]
  for some $h_- \in \cHY^-$.  Defining 
  \begin{equation}
    \label{eq:quantumcorrections}
    f = 
    \sum_{\substack{d \in \NE(\cX): d \ne 0,\\
        c_1(T\cX)\cdot d = 0}} \
    \sum_{e = r+1}^N
    (-1)^{{1 \over 2}w_e + 1} \correlator{\phi^e \psi^{{1 \over 2}w_e-2}}^\cX_{0,1,d} U^d \,
    b_e,
  \end{equation}
  we are done.
\end{proof}

We have seen that the small quantum cohomology of $\cX$ is isomorphic
as a Frobenius algebra to $T_y/z T_y$ where $y = \U(x)$.
Proposition~\ref{pro:RuanCRC} shows that $T_y/z T_y$ is isomorphic as
a Frobenius algebra to
\[
\Big( \HYNovZ, \newQC{c+f} \,\Big).
\]
Once again there is analytic continuation hidden here: the product
$\newQC{c+f}$ is obtained from the product
\[
\varphi_\alpha \newQC{t} \varphi_\beta =
\sum_{\substack{d \in \NE(Y): \\
      d = d_1 \beta_1 + \cdots + d_r \beta_r}}
\correlator{\varphi_\alpha,\varphi_\beta,\varphi^\epsilon}^Y_{0,3,d}
U_1^{d_1} \cdots U_s^{d_s}
\re^{d_1 t_1} \cdots \re^{d_r t_r} \varphi_\epsilon,
\]
where $t = t_1 \varphi_1 + \cdots + t_r \varphi_r \in H^2(Y;\CC)$ and
$|\re^{t_i}|<R_i$ for $s<i\leq r$, by analytic continuation in
$t_{s+1},\ldots,t_r$ followed by the substitution
\begin{align*}
  & t_i = c_i + f_i  &1 \leq i \leq r
\end{align*}
where $f = f_1 \varphi_1 + \cdots + f_r \varphi_r$.  This proves:
\begin{theorem} \label{thm:RuanCRC} Suppose that $\cX$ is
  semi-positive and that conjecture~\ref{conj} holds.  Let
  $f_1,\ldots,f_r \in \CC[\![U_1,\ldots,U_s]\!]$
  be as in proposition~\ref{pro:RuanCRC} and define $c = c_1 \varphi_1
  + \cdots + c_r \varphi_r \in H^2(Y;\CC)$ by $\U(\fun_\cX) = \fun_Y -
  c z^{-1} + O(z^{-2})$.  Then the Frobenius algebra given by the
  small quantum cohomology of $\cX$ is isomorphic to the Frobenius
  algebra obtained from the small quantum cohomology of $Y$ by
  analytic continuation in the exceptional quantum parameters
  $Q_{s+1},\ldots,Q_r$ (if necessary) followed by the
  change-of-variables
  \begin{equation}
    \label{eq:rightcofv}
    Q_i =
    \begin{cases}
      \re^{c_i + f_i} U_i & 1 \leq i \leq s \\
      \re^{c_i + f_i} & s < i \leq r.
    \end{cases}
  \end{equation}

\end{theorem}

The conclusion of Theorem~\ref{thm:RuanCRC} is almost Ruan's original
Crepant Resolution Conjecture, except that the changes-of-variables
\eqref{eq:wrongcofv} and \eqref{eq:rightcofv} differ.  As $f_i = 0$
when $U_1 = \ldots = U_s = 0$, theorem~\ref{thm:RuanCRC} is a
`quantum-corrected' version of Ruan's original conjecture.  The
quantum corrections $f_1,\ldots,f_r$ often vanish --- for example they
vanish whenever $\cX$ is Fano or when $\cX = \big[\CC^n/G\big]$, as
then the sum on the RHS of \eqref{eq:quantumcorrections} is empty.
But $f_1,\ldots,f_r$ do not vanish in general: they are non-zero, for
instance, when $\cX$ is the cotangent bundle $K_{\PP(1,1,3)}$
\cite{Coates:crepant2}.

\section{A Version of the Bryan--Graber Conjecture}
\label{sec:BG}
Suppose now that conjecture~\ref{conj} holds and that $\U:\cHX \to
\cHY$ sends $\cHX^-$ to $\cHY^-$, so that
\[
\U = U_0 + U_1 z^{-1} + \cdots + U_k z^{-k}
\]
for some non-negative integer $k$ and some linear maps $U_i : \HorbX
\to H^\bullet(Y;\CC)$.  In this case $\U$ induces an isomorphism
between the Frobenius manifolds defined by the quantum cohomology of
$\cX$ and the quantum cohomology of $Y$, as we now explain.

Let $x = \cLX \cap \big({-z} + \cHX^-\big)$ and let $y = \U(x)$.  Then
\begin{align*}
  y &= \cLY \cap \U\big({-z} + \cHX^-\big) \\
  &= \cLY \cap \big({-z} + c + \cHY^-\big)
\end{align*}
where $\U(\fun_\cX)= \fun_Y - c z^{-1} + O(z^{-2})$.  Again, write $c
= c_1 \varphi_1 + \cdots + c_r \varphi_r$.  In view of the discussion
in \S\ref{sec:conetoQC}, $\U$ induces an isomorphism between the
Frobenius manifold
\[
\Big( \HorbXNovZ, \QC{\tau} \,\Big)
\]
obtained by taking basepoint $x \in \cLX$ and using opposite subspace
$\cHX^-$, and the Frobenius manifold
\[
\Big( \HYNovZ, \newQC{c+t} \,\Big)
\]
obtained by taking basepoint $y \in \cLY$ and using opposite subspace
$\cHY^-$.  The parameters $\tau \in \HorbXNovZ$ and $t \in \HYNovZ$
here are identified via the diagram
\[
\xymatrix{
  \cLX \cap \big({-z} + z\cHX^-\big) \ar[rrr]^{\U}
  & & & \cLY \cap \big({-z} + c + z\cHY^-\big) \ar[d]^{p} \\
  z \cHX^-/\cHX^- \ar[u]_{J_\cX(\tau,-z)} & \HorbXNovZ \ar[l]^{\cong}
  & \HYNovZ \ar[r]_{\cong} & z\cHY^-/\cHY-
}
\]
so $t = U_0(\tau)$.  Comparing \eqref{eq:newQCYdivisor} with
\eqref{eq:bigQCYdivisor}, we see that the product $\newQC{c + t}$ can
be obtained from the big quantum product $\QC{t}$ on $\HNovY$ by
analytic continuation in the variables $Q_{s+1},\ldots,Q_r$ followed
by the change-of-variables
\begin{equation}
  \label{eq:BGsubst}
  Q_i =
  \begin{cases}
    \re^{c_i} U_i & 1 \leq i \leq s \\
    \re^{c_i} & s < i \leq r.
  \end{cases}
\end{equation}
This proves:

\begin{theorem} \label{thm:BG} Suppose that conjecture~\ref{conj}
  holds and that $\U:\cHX \to \cHY$ sends $\cHX^-$ to $\cHY^-$.  Then
  there is a linear map $U_0:\HorbX \to \HY$ which identifies the
  Frobenius manifold given by the big quantum cohomology
  \eqref{eq:bigQCX} of $\cX$ with the Frobenius manifold obtained from
  the big quantum cohomology \eqref{eq:bigQCY} of $Y$ by analytic
  continuation in the quantum parameters $Q_{s+1},\ldots,Q_r$ (if
  necessary) followed by the substitution \eqref{eq:BGsubst}.  In
  addition, the map $U_0$ preserves the gradings and Poincar\'e
  pairings, sends $\fun_\cX$ to $\fun_Y$, and satisfies $U_0 \circ
  \big(\rho \CR\big) = \({\pi^\star \rho} \, \cup\) \circ U_0$ for every
  untwisted degree-two class $\rho \in H^2(\cX;\CC)$.
\end{theorem}

\noindent The statements about $U_0$ here come from
lemma~\ref{lem:U0}.  As discussed above, if conjecture~\ref{conj}
holds and $\cX$ satisfies the Hard Lefschetz condition\footnote{This
  condition was discovered in \cite{CCIT:crepant1}.} postulated by
Bryan--Graber \cite{Bryan--Graber} then $\U$ automatically sends
$\cHX^-$ to $\cHY^-$.

The conclusion of Theorem \ref{thm:BG} is almost the same as the
Crepant Resolution Conjecture of Bryan and Graber.  They ask that
$U_0:\HorbX \to H^\bullet(Y;\CC)$ agree with $\pi^\star$ on the
untwisted sector $H^\bullet(\cX;\CC) \subset \HorbX$, whereas we only
have that for the subalgebra of $H^\bullet(\cX;\CC)$ generated by
$H^2(\cX;\CC)$.  Furthermore their change-of-variables has $Q_i =
U_i$, $1 \leq i \leq s$, omitting our factor of $\re^{c_i}$, and for
us the substitution $Q_i = \re^{c_i}$, $s<i\leq r$, need not involve
roots of unity\footnote{See conjecture~\ref{conj2} below, however.}.

\section{Quantization and Higher Genus Gromov--Witten Invariants}
\label{sec:highergenus}

So far we have considered genus-zero Gromov--Witten invariants of
$\cX$ and $Y$.  This corresponds to considering the tree-level part of
the topological A-model with target space $\cX$ or $Y$.  But the full
partition function of the topological A-model is also of significant
interest, and this corresponds to the full descendant potential of
$\cX$,
\begin{equation}
  \label{eq:DX}
  \cDX = \exp \bigg( \sum_{g \geq 0} \hbar^{g-1} \cF^g_\cX \bigg),
\end{equation}
or, similarly, to the full descendant potential $\cDY$ of $Y$.  The
quantity $\cF^g_\cX$ in \eqref{eq:DX} is the genus-$g$ descendant
potential of $\cX$: this is defined in the same way as the genus-zero
descendant potential $\cF^0_\cX$ but with integration over the moduli
stack of stable maps to $\cX$ of genus~$g$ rather than genus zero.
The variable $\hbar$ is a formal parameter.  In this section we give a
generalization of our conjecture which applies to Gromov--Witten
invariants of all genera.  Roughly speaking, we conjecture that $\cDY
= \widehat{\U}(\cDX)$, where $\widehat{\U}$ is the \emph{quantization}
of the symplectic transformation $\U$ from conjecture~\ref{conj}.
This idea occurred simultaneously and independently in both
mathematics and physics \citelist{
  \cite{CCIT:crepant1}\cite{Aganagic--Bouchard--Klemm}\cite{Ruan:conjecture}};
it is a consequence of fundamental insights due to Givental
\cite{Givental:quantization} and Witten \cite{Witten}.

Work of Givental
\citelist{\cite{Givental:quantization}\cite{Givental:symplectic}\cite{Coates--Givental:QRRLS}}
and others \cite{Milanov,Tseng,Lee:1,Lee:2,Faber--Shadrin--Zwonkine}
strongly suggests that the full descendant potential $\cDX$ of $\cX$
should be regarded as an element of the Fock space for the geometric
quantization of $\cHX$.  This point of view is described for manifolds
in \cite{Givental:quantization} and extended to orbifolds in
\cite{Tseng}.  The Fock space for $\cX$ consists of certain formal
germs of functions on $\cHX^+$.  We regard $\cDX$, which depends
formally on the variables $\tau_{a,\epsilon}$, $0 \leq \epsilon \leq
N$, $0 \leq a < \infty$ (\emph{c.f.} equation \ref{eq:F0X}), as the
germ of a function on $\cHX^+$ via the dilaton shift
\eqref{eq:dilatonshiftX}.  This makes $\cDX$ into an element of the
Fock space for $\cX$.  In the same way, using the dilaton shift
\eqref{eq:dilatonshiftY}, we regard $\cDY$ as the germ of a function
on $\cHY^+$ and hence as an element of the Fock space for $Y$.

Suppose now that conjecture~\ref{conj} holds.  As we have chosen bases
for $\HorbX$ and $\HY$, we can represent the transformation $\U:\cHX
\to \cHY$ by a matrix $U$ with entries that are Laurent polynomials in
$z$.  Let $U = U_- U_0 U_+$ be the Birkhoff factorization of this matrix,
so that
\begin{align*}
  &U_- = I + U_{-1} z^{-1} + \cdots + U_{-k} z^{-k},\\
  &U_0 = \text{constant diagonal matrix}, \\
  &U_+ = I + U_{1} z + \cdots + U_{l} z^{l},
\end{align*}
for some $k,l>0$.  (The fact that $U_0$ is a constant diagonal matrix,
not a diagonal matrix of Laurent monomials in $z$, follows from
condition (c) in conjecture~\ref{conj}.)\phantom{.} 

\begin{rem}
  The Birkhoff factorization here can easily be computed using row and
  column operations.  For example, as $U = U_- U_0 U_+$ we see that
  $U_+^{-1}$ is the unique matrix of the form $I + A_1 z + \cdots +
  A_m z^m$ such that $U U_+^{-1}$ contains only negative powers of
  $z$.  This can be computed using column operations on $U$.  The
  transformation $A_i$ lowers degree by $2i$, as $\U$ is
  degree-preserving, and hence $A_i$ is nilpotent; $I + A_1 z + \cdots
  + A_m z^m$ is therefore invertible with polynomial inverse.  This
  determines $U_+$.  The matrices $U_-$ and $U_0$ can be determined similarly.
\end{rem}

If we change our choice of bases for $\HorbX$ and $\HY$ then the
factorization
\begin{align*}
  U = U_- U_0 U_+ &&
  \text{becomes} &&
  A U B^{-1} = (A U_- A^{-1})(A U_0 B^{-1}) (B U_+ B^{-1})
\end{align*}
where $A$ and $B$ are appropriate change-of-basis matrices, and so the
factorization defines linear symplectic isomorphisms
\begin{align*}
  \U_- : \cHY \to \cHY, &&
  \U_0 : \cHX \to \cHY, &&
  \U_+ : \cHX \to \cHX, 
\end{align*}
which are independent of our choice of bases.  Let us identify the
Fock space for $\cX$ with the Fock space for $Y$ via the isomorphism
$\U_0: \cHX \to \cHY$.  In this way we regard $\cDX$ as an element of
the Fock space for $Y$; concretely, this means that we regard $\cDX$
as a formal power series in the variables $t_{a,\epsilon}$, $0 \leq
\epsilon \leq N$, $0 \leq a < \infty$ via the identification
$t_{a,\epsilon} \varphi_\epsilon = \U_0(\tau_{a,\mu} \phi_\mu)$.
Consider now the $\CC(\!(z^{-1})\!)$-linear symplectic transformations
$\T_-, \T_+:\cHY \to \cHY$ defined by 
\begin{align*}
  \T_- = \U_-, && \T_+ = \U_0 \U_+ \U_0^{-1}.
\end{align*}
Propositions~5.3 and~7.3 in \cite{Givental:quantization} give formulas
for the \emph{quantizations} $\widehat{\T}_-$, $\widehat{\T}_+$ of $\T_-$
and $\T_+$: these quantizations are endomorphisms of the Fock space for
$Y$.

\begin{conj} \label{conj3}
  Conjecture~\ref{conj} holds, and in addition
  \[
  \cDY  \propto \widehat{\T}_- \widehat{\T}_+ (\cDX)
  \]
  after an appropriate analytic continuation of $\cDX$ and $\cDY$.
  The symbol `$\propto$' here means `is a scalar multiple of'.
\end{conj}

\begin{rem}
  The scalar multiple in conjecture~\ref{conj3} is determined by the
  condition that the genus-one descendant potential of $Y$ vanishes
  when all the $t_{a,\epsilon}$ are zero.  Thus
  conjecture~\ref{conj3} determines the higher-genus Gromov--Witten
  invariants of $\cX$ in terms of those of $Y$.
\end{rem}

\begin{rem} \label{rem:convergence}
  In order for the analytic continuation indicated in
  conjecture~\ref{conj3} to make sense, we need assume some
  convergence of the total descendant potential $\cDY$.  For example,
  if we require that there are strictly positive real numbers $R_i$,
  $s<i \leq r$, such that each $\cF^g_Y$, $g \geq 0$, depends
  analytically on $Q_{s+1},\ldots,Q_r$ in the domain
  \begin{align*}
    |Q_i| < R_i, && s<i \leq r,
  \end{align*}
  then (as above) the Divisor Equation implies that each $\cF^g_Y$ in
  fact depends analytically on $t_{0,1},\ldots,t_{0,r}$ and
  $Q_{s+1},\ldots,Q_r$ in the domain
  \begin{equation*}
    \begin{aligned}
      &|t_{0,i}| < \infty & 1 \leq i \leq s \\
      &|Q_i \re^{t_{0,i}}|<R_i & s<i \leq r.
    \end{aligned}
  \end{equation*}
  This allows us to set $Q_{s+1} = \cdots = Q_r = 1$, defining
  $\cF^{g,\circledast}_Y$, $g \geq 0$, exactly as we defined
  $\cF^\circledast_Y$ above.  We can then use $\cD_Y^\circledast =
  \exp\Big(\sum_{g \geq 0} \cF^{g,\circledast}_Y\Big)$ in place of
  $\cDY$ in conjecture~\ref{conj3}.  But this convergence assumption
  is difficult to check in practice\footnote{Note however that if $Y$
    is a Calabi--Yau $3$-fold then we can use the Divisor, String, and
    Dilaton Equations to express any Gromov--Witten invariant $
    \correlator{\delta_1 \psi^{a_1},\ldots,\delta_n
      \psi^{a_n}}^Y_{g,n,d}$ in terms of the zero-point Gromov--Witten
    invariant $ \correlator{\ }^Y_{g,0,d}$.  It therefore suffices to
    check the convergence assumption in remark~\ref{rem:convergence}
    for the \emph{non-descendant} Gromov--Witten potentials $\cF^g_Y
    \big|_{t_0 = t; t_1 = t_2 = \cdots = 0}$, $g \geq 0$. }, and it
  would be useful to have a higher-genus analog of
  assumption~\ref{convassum}.
\end{rem}

\begin{rem}
  Bryan and Graber have suggested \cite{Bryan--Graber}*{remark~1.8}
  that when $\cX$ satisfies the Hard Lefschetz condition, the
  higher-genus non-descendant Gromov--Witten potentials
  \begin{align*}
    F^g_\cX(\tau) = \cF^g_\cX \big|_{\tau_0 = \tau; \tau_1 = \tau_2 =
      \cdots = 0} && \text{and} &&
    F^g_Y(t) = \cF^g_Y \big|_{t_0 = t; t_1 = t_2 =  \cdots = 0} 
  \end{align*}
  might coincide after analytic continuation in the quantum parameters
  $Q_{s+1},\ldots,Q_r$, the substitution \eqref{eq:BGsubst}, and the
  change-of-variables $t = U_0(\tau)$ from theorem~\ref{thm:BG}.  If
  conjecture~\ref{conj3} and the above convergence assumption hold
  then this is the case.  The Hard Lefschetz condition ensures that
  the transformation $\U_+$ is the identity, and
  conjecture~\ref{conj3} then becomes
  \[
   \cDY^\circledast  \propto \widehat{\U}_- (\cDX).
   \]
   Applying Givental's formula
   \cite{Givental:quantization}*{proposition~5.3} for the operator
   $\widehat{\U_-}$ shows that the non-descendant potentials
   $F^g_\cX(\tau)$ and $F^{g,\circledast}_Y(t)$ are related by
   analytic continuation and a change-of-variables; taking account of
   the substitution \eqref{eq:substagain}, exactly as in
   \S\ref{sec:BG}, shows that $F^g_\cX$ and $F^g_Y$ are related as
   claimed.
\end{rem}

\section{Specializations, $B$-Fields, and Flat Gerbes}
\label{sec:gerbes}

An issue of particular importance for the various Crepant Resolution
Conjectures is to determine the values to which the exceptional
quantum parameters $Q_i$ should be specialized. These values have
physical significance and are referred in the physics literature as
the \emph{$B$-field}. Calculating the correct value of the $B$-field
is a subtle problem even in physics, and although this is understood
in some examples (Hilbert scheme of points, surface singularities, K3
surfaces, etc.) there is not yet a procedure to determine the value of
the $B$-field in general.  One advantage of our approach is that it
gives such a procedure: we can interpret the values of the
specialization (and hence the value of the B-field) as coming from a
shift in basepoint on Givental's cone.  In this section we study this
issue and relate it to the physical point of view on the B-field.
First we propose a further conjecture to constrain the choice of
shift.

\begin{conj} \label{conj2}
  Suppose that conjecture~\ref{conj} holds, so that
  \begin{equation*}
    \U\big(\fun_\cX\big) = \fun_Y - c  z^{-1} + O(z^{-2})
  \end{equation*}
  for some $c \in H^2(Y;\CC)$.  Then in fact $c \in H^2(Y;\QQ
  \sqrt{-1})$.
\end{conj}

\noindent Note that this implies that the quantities $\re^{c_i}$
occurring in theorems~\ref{thm:CCRC},~\ref{thm:RuanCRC}, and~\ref{thm:BG}
are roots of unity.

Now we introduce the notion of Gromov-Witten invariants twisted by a
flat gerbe.  Twisting by a flat gerbe is believed to be the correct
mathematical analog of `turning on a $B$-field' in physics.  The
general construction in the orbifold case has been worked out by
Pan--Ruan--Yin \cite{Pan--Ruan--Yin}. In the smooth case it is
particularly easy. For a smooth manifold $Y$, giving a flat gerbe on
$Y$ is equivalent to giving its \emph{holonomy}, which is a cohomology
class $\theta\in H^2(Y, U(1))$.  Gromov-Witten invariants twisted by
this flat gerbe coincide with the usual Gromov--Witten invariants of
$Y$, but multiplied by a phase factor given by the holonomy:
\begin{equation}
  \label{eq:GWgerbe}
  \correlator{\delta_1 \psi^{a_1},\ldots,\delta_n
    \psi^{a_n}}^{Y,\theta}_{0,n,d} =
  \theta(d)
  \correlator{\delta_1 \psi^{a_1},\ldots,\delta_n
    \psi^{a_n}}^Y_{0,n,d} .
\end{equation}
We will only need the case when $Y$ is smooth, so the reader
unfamiliar with $\theta$-twisted Gromov--Witten invariants can take
\eqref{eq:GWgerbe} as the definition.  It is clear that on smooth
manifolds the set of all $\theta$-twisted Gromov-Witten invariants,
for any flat gerbe $\theta$, contains the same information as the set
of ordinary Gromov-Witten invariants.  The class $c$ in conjecture 4.2
induces a flat gerbe $\theta_{c}$ through the coefficient exact
sequence
\[
\xymatrix{
  0 \ar[r] &
  \sqrt{-1} \, \ZZ \ar[r] &
  \sqrt{-1} \, \RR \ar[rr]^{x \mapsto \exp(2 \pi x)} &&
  U(1) \ar[r] &
  0}.
\]
On the other hand, if $H^3(Y, \sqrt{-1}\,\ZZ)=0$ then any flat gerbe
$\theta$ has a lift $\rho_{\theta} \in H^2(Y;\sqrt{-1}\,\RR)$.

We can define $\theta$-twisted versions $F_{Y,\theta}$,
$F_{Y,\theta}^{\circledast}$, and $\cL_{Y, \theta}$ of $F_{Y}$,
$F_{Y}^{\circledast}$, and $\cL_{Y}$ respectively, by replacing
ordinary Gromov-Witten invariants with $\theta$-twisted Gromov-Witten
invariants.

\begin{lemma}
  Suppose that $\rho_{\theta}$ is a lifting of $\theta$. Then
  multiplication by $e^{\rho_{\theta}/z}$ defines a symplectic
  transformation $\cHY \to \cHY$ such that
  $e^{\rho_{\theta}/z}\cLY=\cL_{Y, \theta}$.
\end{lemma}

\begin{proof}
  Combine the Divisor Equation (see
  \cite{Coates--Givental:QRRLS}*{equation 8}) with \eqref{eq:GWgerbe}.
\end{proof}

\begin{cor}
  If conjectures~\ref{conj} and \ref{conj2} hold then the symplectic
  transformation $\U_c: \cHX \to \cHY$ defined by $\U_c = \re^{c/z}\U$
  satisfies properties (a--d) of conjecture~\ref{conj} and also:
  \begin{align*}
    \U_c (\cLX) = \cL_{Y,\theta_c} &
    & \U_c(\fun_{\cX})=\fun_Y +O(z^{-2}).
  \end{align*}
\end{cor}

Recall from \S\S\ref{sec:CCRC}--\ref{sec:BG} that the cohomology class
$c \in H^2(Y;\CC)$ defined by $\U(\fun_\cX) = \fun_Y - c z^{-1} +
O(z^{-2})$ gives rise to the values $\re^{c_i}$ to which the
exceptional quantum parameters are specialized: in other words $\U$
picks out the $B$-field.  It does this because $c$ produces the `shift
in basepoint' $\newQC{t} \rightsquigarrow \newQC{t+c}$ visible, for
instance, in equation \eqref{eq:CCRClimit}.  If we repeat the analysis
of \S\S\ref{sec:CCRC}--\ref{sec:BG} but using the symplectic
transformation $\U_c$ rather than $\U$ then on the one hand we should
replace each $\re^{c_i}$ by $1$ (because $\U_c(\fun_{\cX})=\fun_Y
+O(z^{-2})$ and so now there is no shift in basepoint) and on the
other hand we should replace the quantum cohomology of $Y$ by the
$\theta_c$-twisted quantum cohomology (because we consider the
submanifold-germ $\cL_{Y,\theta}$ not $\cLY$).  In other words, our
conjectures predict the emergence of a flat gerbe $\theta_{c}$.  We
can use this to give a very clean version of the Cohomological Crepant
Resolution Conjecture:

\begin{conj*}[Modified CCRC] 
  There is a flat gerbe $\theta$ on $Y$ such that the Chen--Ruan
  product $\CR$ on $\HorbX$ can be obtained from the $\theta$-twisted
  small quantum product for $Y$ by analytic continuation in the
  quantum parameters $Q_{s+1},\ldots,Q_r$ (if necessary) followed by
  the substitution
  \begin{equation*}
    Q_i =
    \begin{cases}
      0 & 1 \leq i \leq s \\
      1 & s < i \leq r.
    \end{cases}
  \end{equation*}
\end{conj*}

\noindent Conjectures~\ref{conj} and~\ref{conj2} together imply the
Modified CCRC with $\theta = \theta_c$.  We can give a
similarly-improved version of Ruan's Crepant Resolution Conjecture,
which again follows from Conjectures~\ref{conj} and~\ref{conj2}:

\begin{conj*}(Modified CRC) Suppose that $\cX$ is semi-positive.  Then
  there is a flat gerbe $\theta$ over $Y$ and a choice of elements
  $f_1,\ldots,f_r \in \CC[\![U_1,\ldots,U_s]\!]$ such that $f_i = 0$
  when $U_1 = \cdots = U_s = 0$, such that the class $f = f_1
  \varphi_1 + \cdots + f_r \varphi_r$ is exceptional, and such that
  the Frobenius algebra given by the small quantum cohomology of $\cX$
  is isomorphic to the Frobenius algebra obtained from the
  $\theta$-twisted small quantum cohomology of $Y$ by analytic
  continuation in the exceptional quantum parameters
  $Q_{s+1},\ldots,Q_r$ (if necessary) followed by the
  change-of-variables
  \begin{equation*}
    Q_i =
    \begin{cases}
      \re^{f_i} U_i & 1 \leq i \leq s \\
      \re^{ f_i} & s < i \leq r.
    \end{cases}
  \end{equation*}
\end{conj*}

\noindent The corrections $f_i$ here and in \eqref{eq:rightcofv} are
an example of what physicists call a `mirror map'.

\section*{Appendix: Proofs of Analyticity Results}

\begin{applemma1} 
  The descendant potential $\cF^0_\cX$, which is a formal power series
  in the variables $U_1,\ldots,U_s$ and
  $\tau_{a,\epsilon}$, $0 \leq \epsilon \leq N$, $0 \leq a < \infty$,
  in fact depends analytically on $\tau_{0,1},\ldots,\tau_{0,s}$ in
  the domain $\CC^s$.
\end{applemma1}

\begin{proof}
  Set
  \begin{align*}
    & \tau_{0,\rm two}  = \tau_{0,1} \phi_1 + \cdots + \tau_{0,s}
    \phi_s, \\
    & \big[\phi_{e_1} \psi^{a_1},\ldots,\phi_{e_k} \psi^{a_k} \big]^\cX_{0,d}
    = \sum_{n \geq 0} {1 \over n!}
    \correlator{\phi_{e_1} \psi^{a_1},\ldots,\phi_{e_k} \psi^{a_k},
      \tau_{0,\rm two},\ldots,\tau_{0,\rm two}}^{\cX}_{0,n+k,d}, \\
  & \ccorrelator{\phi_{e_1} \psi^{a_1},\ldots,\phi_{e_k}
    \psi^{a_k}}^\cX_0
   = \sum_{d \in \NE(\cX)}
  \big[\phi_{e_1} \psi^{a_1},\ldots,\phi_{e_k} \psi^{a_k} \big]^\cX_{0,d}
  \, U^d ,
  \end{align*}
  and call the quantity $\big[\phi_{e_1} \psi^{a_1},\ldots,\phi_{e_k}
  \psi^{a_k} \big]^\cX_{0,d}$ a \emph{$k$-point descendant}.  We need
  to show that each $k$-point descendant is an entire function of
  $\tau_{0,1},\ldots,\tau_{0,s}$; let us call this property
  \emph{entireness}.  The Topological Recursion Relations
  \cite{Tseng}*{\S2.5.5} express any $k$-point descendant
  $\big[\phi_{e_1} \psi^{a_1},\ldots,\phi_{e_k} \psi^{a_k}
  \big]^\cX_{0,d}$ with $k \geq 3$ and at least one non-zero $a_i$ as
  a linear combination of $l$-point descendants with $l<k$.  Thus we
  need to establish entireness for $k$-point descendants with $k=0$,
  $k=1$, $k=2$, or $k$ arbitrary but $a_1 = \cdots = a_k = 0$.  The
  cases $k=0$ and $k$ arbitrary but $a_1 = \cdots = a_k = 0$ follow
  from the entireness of the potential $F_\cX$ (see equation
  \ref{eq:bigQCXdivisor}).  The cases $k=1$ and $k = 2$ but $a_2 = 0$
  follow from proposition~\ref{pro:divisorX}.  The remaining case ---
  $k=2$ but $a_1, a_2 \ne 0$ --- follows from the WDVV-like identity
  \begin{equation}
    \label{eq:WDVVlike}
    \CCorrelator{{\phi_\alpha \over z - \psi},1,
      {\phi_\beta \over w - \psi}}^\cX_0 =
    \CCorrelator{{\phi_\alpha \over z - \psi},1,
      \phi_\epsilon}^\cX_0
    \CCorrelator{\phi^\epsilon,1,
      {\phi_\beta \over w - \psi}}^\cX_0
  \end{equation}
  and the String Equation
  \begin{equation}
    \label{eq:string}
    \begin{aligned}
      &\CCorrelator{{\phi_\alpha \over z - \psi},1,
        {\phi_\beta \over w - \psi}}^\cX_0 =
      {1 \over zw} \big(\phi_\alpha,\phi_\beta)_\cX +
      \Big({1 \over z} + {1 \over w}\Big)
      \CCorrelator{{\phi_\alpha \over z - \psi},
        {\phi_\beta \over w - \psi}}^\cX_0,\\
      &\CCorrelator{{\phi_\alpha \over z - \psi},1,
        \phi_\epsilon}^\cX_0
      =     {1 \over z} \big(\phi_\alpha,\phi_\epsilon)_\cX +
      {1 \over z} \CCorrelator{{\phi_\alpha \over z - \psi},
        \phi_\epsilon}^\cX_0.
    \end{aligned}
  \end{equation}
  Thus $\cF^0_\cX$ depends analytically on
  $\tau_{0,1},\ldots,\tau_{0,s}$ in the domain $\CC^s$.
\end{proof}

\begin{applemma2} Assume that convergence
  assumption \ref{convassum} holds.  Then the descendant potential
  $\cF^0_Y$, which is a formal power series in the variables
  $Q_1,\ldots,Q_r$ and $t_{a,\epsilon}$, $0 \leq
  \epsilon \leq N$, $0 \leq a < \infty$, in fact depends analytically
  on $t_{0,1},\ldots,t_{0,r}$ and
  $Q_{s+1},\ldots,Q_r$ in the domain
  \begin{equation}
    \label{eq:secondregionagain}
    \begin{aligned}
      &|t_{0,i}| < \infty & 1 \leq i \leq s \\
      &|Q_i \re^{t_{0,i}}|<R_i & s<i \leq r.
    \end{aligned}
  \end{equation}
\end{applemma2}

\begin{proof}
  This is very similar to the proof of the preceding lemma.  As
  before, set
  \begin{align*}
    & t_{0,\rm two}  = t_{0,1} \varphi_1 + \cdots + t_{0,r}
    \varphi_r, \\
    & \big[\varphi_{e_1} \psi^{a_1},\ldots,\varphi_{e_k} \psi^{a_k} \big]^Y_{0,d}
    = \sum_{n \geq 0} {1 \over n!}
    \correlator{\varphi_{e_1} \psi^{a_1},\ldots,\varphi_{e_k} \psi^{a_k},
      t_{0,\rm two},\ldots,t_{0,\rm two}}^{Y}_{0,n+k,d}, \\
  & \ccorrelator{\varphi_{e_1} \psi^{a_1},\ldots,\varphi_{e_k}
    \psi^{a_k}}^Y_0
   = \sum_{d \in \NE(Y)}
  \big[\varphi_{e_1} \psi^{a_1},\ldots,\varphi_{e_k} \psi^{a_k} \big]^Y_{0,d}
  \, Q^d .
  \end{align*}
  We need to show that, for each choice of $d_1,\ldots,d_s \in \QQ$
  with $d_i \geq 0$, the coefficient of $Q_1^{d_1}\cdots Q_s^{d_s}$ in
  $\ccorrelator{\varphi_{e_1} \psi^{a_1},\ldots,\varphi_{e_k}
    \psi^{a_k}}^Y_0$ defines an analytic function of
  $t_{0,1},\ldots,t_{0,r}$ and
  $Q_{s+1},\ldots,Q_r$ in the domain
  \eqref{eq:secondregionagain}.  Let us call this property
  \emph{analyticity} of $\ccorrelator{\varphi_{e_1}
    \psi^{a_1},\ldots,\varphi_{e_k} \psi^{a_k}}^Y_0$.

  The Topological Recursion Relations \cite{Cox--Katz}*{lemma~10.2.2}
  show that it suffices to establish analyticity of
  $\ccorrelator{\varphi_{e_1} \psi^{a_1},\ldots,\varphi_{e_k}
    \psi^{a_k}}^Y_0$ in the cases where $k=0$, $k=1$, $k=2$, or $k$
  arbitrary but $a_1 = \cdots = a_k = 0$.  The cases $k=0$ and $k$
  arbitrary but $a_1 = \cdots = a_k = 0$ follow from convergence
  assumption~\ref{convassum} (see the discussion above equation
  \ref{eq:firstregion}).  The cases $k=1$ and $k = 2$ with $a_1, a_2
  \ne 0$ follow from the case $k=2$ but $a_2 = 0$, in view of
  identities \eqref{eq:WDVVlike}, \eqref{eq:string}, and the String
  Equation
  \[
  \CCorrelator{{\varphi_\alpha \over z - \psi},1}^Y_0 =
  {1 \over z}
  \big(\varphi_\alpha,t_{0,\rm two})_Y + {1 \over z}
  \CCorrelator{{\varphi_\alpha \over z - \psi}}^Y_0.
  \]
  It remains to establish the analyticity of
  $\ccorrelator{{\phi_\alpha \over z - \psi},\phi_\beta}^Y_0$ for all
  $\alpha$ and $\beta$; this holds as these quantities are solutions
  to a system of differential equations (the `quantum differential
  equations' \cite{Cox--Katz}*{proposition~10.2.1}) with coefficients
  which are known, by convergence assumption~\ref{convassum}, to be
  analytic functions defined in the domain
  \eqref{eq:secondregionagain}.  The lemma is proved.
\end{proof}


\begin{bibdiv}
\begin{biblist}

\bib{Aganagic--Bouchard--Klemm}{article}{
    author={Aganagic, Mina},
    author={Bouchard, Vincent},
    author={Klemm, Albrecht},
    title={Topological Strings and (Almost) Modular Forms},
    eprint={hep-th/0607100},
  }

\bib{AGV:1}{article}{
    author={Abramovich, Dan},
    author={Graber, Tom},
    author={Vistoli, Angelo},
    title={Algebraic orbifold quantum products},
    conference={
       title={Orbifolds in mathematics and physics},
       address={Madison, WI},
       date={2001},
    },
    book={
       series={Contemp. Math.},
       volume={310},
       publisher={Amer. Math. Soc.},
       place={Providence, RI},
    },
    date={2002},
     pages={1--24},
     review={\MR{1950940 (2004c:14104)}},
 }

 \bib{AGV:2}{article}{
    author={Abramovich, Dan},
    author={Graber, Tom},
   author={Vistoli, Angelo},
   title={Gromov--Witten theory of Deligne--Mumford stacks},
   date={2006},
   eprint={arXiv:math.AG/0603151},
 }

\bib{Aspinwall--Greene--Morrison}{article}{
   author={Aspinwall, Paul S.},
   author={Greene, Brian R.},
   author={Morrison, David R.},
   title={Calabi-Yau moduli space, mirror manifolds and spacetime topology
   change in string theory},
   journal={Nuclear Phys. B},
   volume={416},
   date={1994},
   number={2},
   pages={414--480},
   issn={0550-3213},
   review={\MR{1274435 (95i:32027)}},
}

\bib{Barannikov}{article}{
   author={Barannikov, Serguei},
   title={Quantum periods. I. Semi-infinite variations of Hodge structures},
   journal={Internat. Math. Res. Notices},
   date={2001},
   number={23},
   pages={1243--1264},
   issn={1073-7928},
   review={\MR{1866443 (2002k:32017)}},
}

\bib{Beilinson--Bernstein--Deligne}{article}{
   author={Be{\u\i}linson, A. A.},
   author={Bernstein, J.},
   author={Deligne, P.},
   title={Faisceaux pervers},
   language={French},
   conference={
      title={Analysis and topology on singular spaces, I},
      address={Luminy},
      date={1981},
   },
   book={
      series={Ast\'erisque},
      volume={100},
      publisher={Soc. Math. France},
      place={Paris},
   },
   date={1982},
   pages={5--171},
   review={\MR{751966 (86g:32015)}},
}

\bib{Bryan--Graber}{article}{
  author = {Bryan, Jim},
  author = {Graber, Tom},
  title = {The Crepant Resolution Conjecture},
  eprint = {arXiv:math.AG/0610129},
}

\bib{Bryan--Graber--Pandharipande}{article}{
  author = {Bryan, Jim},
  author = {Graber, Tom},
  author = {Pandharipande, Rahul},
  title = {The orbifold quantum cohomology of $\CC^2/\ZZ_3$ and Hurwitz--Hodge
    integrals},
  eprint = {arXiv:math.AG/0510335}
}

\bib{Chen--Ruan:orbifold}{article}{
   author={Chen, Weimin},
   author={Ruan, Yongbin},
   title={A new cohomology theory of orbifold},
   journal={Comm. Math. Phys.},
   volume={248},
   date={2004},
   number={1},
   pages={1--31},
   issn={0010-3616},
   review={\MR{2104605 (2005j:57036)}},
}

\bib{Chen--Ruan:GW}{article}{
   author={Chen, Weimin},
   author={Ruan, Yongbin},
   title={Orbifold Gromov--Witten theory},
   conference={
      title={Orbifolds in mathematics and physics},
      address={Madison, WI},
      date={2001},
   },
   book={
      series={Contemp. Math.},
      volume={310},
      publisher={Amer. Math. Soc.},
      place={Providence, RI},
   },
   date={2002},
   pages={25--85},
   review={\MR{1950941 (2004k:53145)}},
}

\bib{Coates}{article}{
  title = {Givental's Lagrangian Cone and $S^1$-Equivariant Gromov--Witten
    Theory},
  author = {Coates, Tom},
  eprint = {arXiv:math.AG/0607808},
}

\bib{Coates:crepant2}{article}{
  title = {Wall-Crossings in Toric Gromov--Witten Theory II: Local
    Examples},
  author = {Coates, Tom},
  eprint = {arXiv:0804.2592v1}
}

\bib{CCIT:crepant1}{article}{
  title = {Wall-Crossings in Toric Gromov--Witten Theory I: Crepant
    Examples},
  author = {Coates, Tom},
  author = {Corti, Alessio},
  author = {Iritani, Hiroshi},
  author = {Tseng, Hsian-Hua},
  eprint = {arXiv:math.AG/0611550}
}

 \bib{CCLT}{article}{
   title={The Quantum Orbifold Cohomology of Weighted Projective Space},
   author={Coates, Tom},
   author={Corti, Alessio},
   author={Lee, Yuan-Pin},
   author={Tseng, Hsian-Hua},
   eprint={arXiv:math.AG/0608481},
 }

\bib{Coates--Givental:QRRLS}{article}{
  author={Coates, Tom},
  author={Givental, Alexander},
  title={Quantum Riemann-Roch, Lefschetz and Serre},
  journal={Ann. of Math. (2)},
  volume={165},
  date={2007},
  number={1},
  pages={15--53},
  issn={0003-486X},
  review={\MR{2276766}},
}

\bib{Cox--Katz}{book}{
   author={Cox, David A.},
   author={Katz, Sheldon},
   title={Mirror symmetry and algebraic geometry},
   series={Mathematical Surveys and Monographs},
   volume={68},
   publisher={American Mathematical Society},
   place={Providence, RI},
   date={1999},
   pages={xxii+469},
   isbn={0-8218-1059-6},
   review={\MR{1677117 (2000d:14048)}},
}

\bib{Dubrovin}{article}{
   author={Dubrovin, Boris},
   title={Geometry of $2$D topological field theories},
   conference={
      title={Integrable systems and quantum groups},
      address={Montecatini Terme},
      date={1993},
   },
   book={
      series={Lecture Notes in Math.},
      volume={1620},
      publisher={Springer},
      place={Berlin},
   },
   date={1996},
   pages={120--348},
   review={\MR{1397274 (97d:58038)}},
}

\bib{Faber--Shadrin--Zwonkine}{article}{
    title = {Tautological relations and the r-spin Witten conjecture},
    author = {Faber, Carel},
    author = {Shadrin, Sergey},
    author = {Zvonkine, Dimitri},
    eprint = {arXiv:math/0612510}
  }

\bib{Fulton--Pandharipande}{article}{
   author={Fulton, W.},
   author={Pandharipande, R.},
   title={Notes on stable maps and quantum cohomology},
   conference={
      title={Algebraic geometry---Santa Cruz 1995},
   },
   book={
      series={Proc. Sympos. Pure Math.},
      volume={62},
      publisher={Amer. Math. Soc.},
      place={Providence, RI},
   },
   date={1997},
   pages={45--96},
   review={\MR{1492534 (98m:14025)}},
}

\bib{Givental:homological}{article}{
   author={Givental, Alexander B.},
   title={Homological geometry. I. Projective hypersurfaces},
   journal={Selecta Math. (N.S.)},
   volume={1},
   date={1995},
   number={2},
   pages={325--345},
   issn={1022-1824},
   review={\MR{1354600 (97c:14052)}},
}

\bib{Givental:quantization}{article}{
  author={Givental, Alexander B.},
  title={Gromov-Witten invariants and quantization of quadratic
    Hamiltonians},
  language={English, with English and Russian summaries},
  journal={Mosc. Math. J.},
  volume={1},
  date={2001},
  number={4},
  pages={551--568, 645},
  issn={1609-3321},
  review={\MR{1901075 (2003j:53138)}},
}

 \bib{Givental:symplectic}{article}{
   author={Givental, Alexander B.},
   title={Symplectic geometry of Frobenius structures},
   conference={
     title={Frobenius manifolds},
   },
   book={
     series={Aspects Math., E36},
     publisher={Vieweg},
     place={Wiesbaden},
   },
   date={2004},
   pages={91--112},
   review={\MR{2115767 (2005m:53172)}},
 }

\bib{Hertling}{book}{
   author={Hertling, Claus},
   title={Frobenius manifolds and moduli spaces for singularities},
   series={Cambridge Tracts in Mathematics},
   volume={151},
   publisher={Cambridge University Press},
   place={Cambridge},
   date={2002},
   pages={x+270},
   isbn={0-521-81296-8},
   review={\MR{1924259 (2004a:32043)}},
}

\bib{Hertling--Manin}{article}{
   author={Hertling, C.},
   author={Manin, Yu.},
   title={Weak Frobenius manifolds},
   journal={Internat. Math. Res. Notices},
   date={1999},
   number={6},
   pages={277--286},
   issn={1073-7928},
   review={\MR{1680372 (2000j:53117)}},
}

\bib{Keel--Mori}{article}{
   author={Keel, Se{\'a}n},
   author={Mori, Shigefumi},
   title={Quotients by groupoids},
   journal={Ann. of Math. (2)},
   volume={145},
   date={1997},
   number={1},
   pages={193--213},
   issn={0003-486X},
   review={\MR{1432041 (97m:14014)}},
}		

\bib{Lee:1}{article}{
    title = {Invariance of tautological equations I: conjectures and applications},
    author = {Lee, Y.-P.},
    eprint = {arXiv:math/0604318},
}

\bib{Lee:2}{article}{
    title = {Invariance of tautological equations II: Gromov--Witten theory},
    author = {Lee, Y.-P.},
    eprint = {arXiv:math/0605708},
}

\bib{Manin}{book}{
   author={Manin, Yuri I.},
   title={Frobenius manifolds, quantum cohomology, and moduli spaces},
   series={American Mathematical Society Colloquium Publications},
   volume={47},
   publisher={American Mathematical Society},
   place={Providence, RI},
   date={1999},
   pages={xiv+303},
   isbn={0-8218-1917-8},
   review={\MR{1702284 (2001g:53156)}},
}

\bib{Milanov}{article}{
  title = {Gromov--Witten Theory of $\mathbb{CP}^1$ and Integrable Hierarchies},
  author = {Milanov, Todor E.},
  eprint = {arXiv:math-ph/0605001},
}

\bib{Pan--Ruan--Yin}{article}{
  title = {Gerbes and twisted orbifold quantum cohomology},
  author = {Pan, Jianzhong},
  author = {Ruan, Yongbin},
  author = {Yin, Xiaoqin},
  eprint = {arXiv:math.AG/0504369},
}

\bib {Perroni}{article}{
  title = {Orbifold Cohomology of ADE-singularities},
  author = {Perroni, Fabio},
  eprint = {arXiv:math.AG/0510528},
}

\bib{Ruan:firstconjecture}{article}{
   author={Ruan, Yongbin},
   title={The cohomology ring of crepant resolutions of orbifolds},
   conference={
      title={Gromov-Witten theory of spin curves and orbifolds},
   },
   book={
      series={Contemp. Math.},
      volume={403},
      publisher={Amer. Math. Soc.},
      place={Providence, RI},
   },
   date={2006},
   pages={117--126},
   review={\MR{2234886 (2007e:14093)}},
}

\bib{Ruan:conjecture}{article}{
  author = {Ruan, Yongbin},
  status = {unpublished},
}

 \bib{Tseng}{article}{
     author = {Tseng, Hsian-Hua},
     title = {Orbifold Quantum Riemann--Roch, Lefschetz and Serre},
     eprint = {arXiv:math.AG/0506111},
 }

\bib{Witten}{article}{
  eprint = {arXiv:hep-th/9306122},
   title = {Quantum Background Independence In String Theory},
   author = {Witten, Edward},
}

\end{biblist}
\end{bibdiv}
\end{document}